\documentclass[pdflatex,sn-mathphys-num]{sn-jnl}

\usepackage{amsmath,amssymb,amsfonts,amsthm,amscd,mathtools}
\usepackage{stmaryrd,mathrsfs}
\usepackage{enumerate}
\usepackage{tikz}
\usepackage{tikz-cd}
\usetikzlibrary{arrows,patterns}
\usepackage{url}

\newtheorem{theorem}{Theorem}[section]
\newtheorem{lemma}[theorem]{Lemma}
\newtheorem{prop}[theorem]{Proposition}

\newtheorem{defn}[theorem]{Definition}
\newtheorem{example}[theorem]{Example}
\newtheorem{remark}[theorem]{Remark}
\newtheorem{definition}{Definition}

\numberwithin{equation}{section}

\newcommand{\Om}{\Omega}
\newcommand{\bas}{\mathrm{bas}}
\newcommand{\dd}{\mathrm{d}}
\newcommand{\Lie}{\mathrm{Lie}}
\newcommand{\rrto}{\rightrightarrows}
\newcommand{\BSS}{\mathrm{BSS}}

\begin{document}

\title[Equivariant basic cohomology of Lie groupoids]{Equivariant basic cohomology of Lie groupoids}

\author[1]{\fnm{Fengyu} \sur{Jiang}}\email{jfyscu@foxmail.com}

\author*[1]{\fnm{Yang} \sur{Yang}}\email{yangyang7121@outlook.com}

\author[1]{\fnm{Bohui} \sur{Chen}}\email{chenbohui@scu.edu.cn}

\affil[1]{\orgdiv{Department of Mathematics},
\orgname{Sichuan University},
\orgaddress{\city{Chengdu}, \postcode{610065}, \country{China}}}

\abstract{This paper develops equivariant basic cohomology for Lie groupoids equipped with weak actions of Lie groups. The weak action is encoded by a Kan fibration over the classifying groupoid, and the basic complex of the fiber is shown to carry the structure needed for Weil and Cartan models. The construction is compared with Bott--Shulman--Stasheff cohomology, where the equivariant theory is obtained from the quotient groupoid. For orbifolds, basic forms are interpreted as orbifold differential forms, and the resulting equivariant basic cohomology is used to formulate differential-geometric constructions such as equivariant integration and localization. The paper also studies the induced weak action on the inertia groupoid and uses it to define an equivariant refinement of the Chen--Ruan cohomology ring. In this framework the sectorwise equivariant cohomology, obstruction bundle, equivariant Euler class, Gysin maps and three-point functions are assembled into an equivariant Chen--Ruan product whenever the corresponding pairing is nondegenerate.}



\maketitle

\section{Introduction}

Let
\[
        X=(X_1\rightrightarrows X_0)
\]
be a Lie groupoid with source and target maps
\(
        s,t:X_1\longrightarrow X_0 .
\)
The complex of basic forms on \(X\) is defined by
\[
        \Omega^k_{\mathrm{bas}}(X)
        :=
        \left\{
        \alpha\in \Omega^k(X_0)
        \ \middle|\
        s^*\alpha=t^*\alpha
        \right\}.
\]
The differential is the ordinary de Rham differential on \(X_0\), which
preserves the above condition. Thus one obtains a cochain complex
\[
        \bigl(\Omega^\bullet_{\mathrm{bas}}(X),d\bigr).
\]
For an action groupoid
\[
        G\ltimes M\rightrightarrows M,
\]
this recovers the classical complex of basic forms for the action, namely
forms on \(M\) which are invariant and horizontal with respect to the
infinitesimal \(G\)-action. More generally, the complex
\(\Omega^\bullet_{\mathrm{bas}}(X)\) may be viewed as the de Rham complex of
forms descending to the orbit space of the groupoid.

It is useful to contrast this with the Bott--Shulman--Stasheff complex $\Omega^\bullet_\BSS(X)$. The BSS complex is the total de Rham complex of the nerve $X_\bullet$, namely forms on all $X_p$'s with the de Rham and simplicial differentials. It computes the de Rham cohomology of the groupoid/classifying stack rather than the ordinary basic/orbit-space theory. Recent expositions describe the BSS complex as the double complex $\Omega^\ast(X_\bullet)$, whose cohomology is the de Rham cohomology of the Lie groupoid.

To summarize, the basic complex $\Omega^\bullet_\bas(X)$
is the Lie-groupoid analogue of basic forms for actions and foliations. For proper Lie groupoids, it models differential forms on the orbit space, especially when the orbit space is viewed diffeologically or as an orbispace. The main references are Pflaum--Posthuma--Tang~\cite{PflaumPosthumaTang2014} and Watts~\cite{Watts2022}. The BSS complex, by contrast, models the de Rham cohomology of the stack/classifying space presented by the groupoid.

Let $G$ be a Lie group.
There are several  related notions of a group action on a Lie groupoid, depending
on how strictly one wants the action to be represented.
The most direct notion is the strict one.  Let
\[
        X=(X_1\rightrightarrows X_0)
\]
be a Lie groupoid and let \(G\) be a Lie group.  A strict \(G\)-action on \(X\)
consists of smooth \(G\)-actions on \(X_0\) and \(X_1\) such that, for every
\(g\in G\), the induced maps
\[
        g:X_0\to X_0,
        \qquad
        g:X_1\to X_1
\]
define a Lie groupoid automorphism of \(X\).  Equivalently, the source, target,
unit, inverse, and multiplication maps of \(X\) are all \(G\)-equivariant.  Lie
groupoids equipped with such compatible \(G\)-actions are often called
\(G\)-groupoids.

However, strict actions are not sufficiently invariant under Morita
equivalence.  If a Lie groupoid is regarded as a presentation of a
differentiable stack or an orbifold, replacing it by a Morita-equivalent
presentation may destroy the strictness of the action.  Thus, in stacky or
orbifold geometry, the natural notion of symmetry is generally weaker and
should be formulated in a Morita-invariant way.

One standard approach is to use Hilsum--Skandalis bibundles, or generalized
morphisms, between Lie groupoids.  In this language, a weak action of \(G\) on a
Lie groupoid \(X\) is not required to be given by strict automorphisms of a
fixed presentation.  Instead, the elements of \(G\) act by Morita
self-equivalences of \(X\), and the group law is implemented up to coherent
isomorphism of bibundles.  This is the formulation naturally compatible with
the bicategory of Lie groupoids, principal bibundles, and bibundle maps.
Equivalently, from the viewpoint of differentiable stacks, a weak \(G\)-action
is an action on the stack presented by \(X\).  

The classifying stack \(BG\)
classifies principal \(G\)-bundles, and a stack over \(BG\) may be interpreted
as a quotient object for a \(G\)-action.  In the Lie groupoid presentation used
in this paper, we encode a weak \(G\)-action by a Kan fibration
\[
        f:Y\longrightarrow BG .
\]
Its fiber
\[
        X=\ker(f)
\]
is the Lie groupoid being acted on.  Thus \(Y\) should be thought of as the
total groupoid presenting the quotient of \(X\) by the weak \(G\)-action. 
 We formally denote $Y$ by $X\rtimes BG$.
The Kan fibration formalism is compatible with the approach using Hilsum-Skandalis morphisms~\cite{HilsumSkandalis1987} and the approach using generalized morphisms~\cite{delHoyo2013,MoerdijkMrcun2003}.

This point of view is standard in stack-theoretic treatments of group actions.
Strict group actions on Lie groupoids are discussed in the language of
\(G\)-groupoids or compatible groupoid actions.  The Morita-invariant
interpretation is developed using differentiable stacks,bibundles and Kan fibrations.  Useful
references include Moerdijk--Mr\v{c}un~\cite{MoerdijkMrcun2003} for Lie groupoids and foliations,
Mackenzie~\cite{Mackenzie2005} for Lie groupoids and Lie algebroids, Behrend--Xu~\cite{BehrendXu2011} for differentiable
stacks and \(BG\), Lerman~\cite{Lerman2010} for orbifolds as stacks, Blohmann~\cite{Blohmann2008} for stacky Lie
groups and smooth principal bibundles, Bursztyn--Noseda--Zhu~\cite{BursztynNosedaZhu2020} for principal actions of stacky
Lie groupoids, Ginot--Noohi~\cite{GinotNoohi2012} for weak group actions on stacks, Chen--Du--Jiang~\cite{ChenDuJiangSubmitted}
 and Du Li~\cite{DuLi2014} for Kan fibrations.

Consider a weak $G$ action on a Lie groupoid $X$. The slogan is the $G$ equivariant $\mathcal H$ cohomology  $\mathcal H^\ast_G(X)$ of $X$ is 
\begin{equation}\label{E_SloganofH}
\mathcal H^\ast_G(X)=\mathcal H^\ast(X\rtimes BG).
\end{equation}
 For instance, when $\mathcal H$ is the BSS cohomology $H^\ast_\BSS$, this is verified for certain important cases, and, for general $X$,
$H^\ast_{G,\BSS}(X)$ is defined to be $H^\ast_{\BSS}(X\rtimes BG) $ once the existence of $X\rtimes BG$ is proved. In particular,
\[
H^\ast_{G,\BSS}(\bullet)
:=H^\ast_G(\bullet)=H^\ast_\BSS(BG).
\]
Here we use the fact $\bullet\rtimes BG=BG$. However, 
\[
H^\ast_{G,\bas}(\bullet)
:=H^\ast_G(\bullet)\not=H^\ast_\bas(BG),
\]
hence the slogan \eqref{E_SloganofH} does not apply to $H^\ast_{G,\bas}$. 

It is known that $\Omega^\bullet_\bas$ is Morita invariant. Hence it is a reasonable object for Lie groupoids. In particular, when $X$ is an orbifold,  
the complex of basic forms has a direct differential-geometric meaning.  For a
smooth manifold \(M\), the complex of basic forms is de Rham complex \(\Omega^\bullet(M)\) 
of differential forms, which is widely used such as defining symplectic forms, volume forms,
characteristic forms, de Rham representatives, and integration over \(M\).
For an orbifold, basic forms are  the orbifold differential forms. Thus
\[
        \Omega^\bullet_{\bas}(X)
\]
plays the role of \(\Omega^\bullet(M)\) for an orbifold.
This complex is not merely a technical subcomplex.  It is the natural complex
for differential geometry on the orbifold.  For example, if \(X\) is compact
and oriented, and
\[
        \alpha\in \Omega^{\dim X}_{\bas}(X),
\]
then \(\alpha\) has a well-defined orbifold integral
\[
        \int_X^{\mathrm{orb}}\alpha .
\]
Therefore, equivariantizing the basic complex is meaningful in its own right.
If a compact Lie group \(G\) acts weakly on an orbifold \(X\), then a
\(G^\ast\)-algebra structure on
\[
        \Omega^\bullet_{\bas}(X)
\]
allows one to define equivariant orbifold differential forms, equivariant
characteristic classes, equivariant symplectic forms, and equivariant
integration and localization formulas. 

This theory should be distinguished from the Bott--Shulman--Stasheff theory.
The BSS complex is built from differential forms on the full nerve of the
groupoid and is more stacky, or classifying-space-like.  By contrast, the basic
complex is the direct analogue of the ordinary de Rham complex of a manifold.
Thus the equivariant cohomology of
\[
        \Omega^\bullet_{\bas}(X)
\]
is the natural setting for differential-geometric constructions on orbifolds.

Chen--Ruan cohomology ring is defined on orbifold and is also in this differential-geometric flavor. The underlying vector space of Chen--Ruan cohomology is still differential-geometric: it consists of de Rham cohomology classes of orbifold sectors, with degree shifts given by the age grading.
The product is also geometric. It is defined using the double or triple inertia orbifold, evaluation maps, the obstruction bundle, its Euler class, and orbifold integration. In this sense, Chen--Ruan cohomology is not merely a homotopy-theoretic or classifying-space construction; it is a geometric cohomology theory built from orbifold forms, sector geometry, obstruction bundles, and integration.

Therefore,
equivariantizing Chen--Ruan cohomology is also natural: one equivariantizes the sectorwise basic forms, the obstruction bundle, the Euler class, the Gysin maps, and the three-point function. The result is an equivariant refinement of the orbifold intersection theory encoded by the Chen--Ruan ring. Existing equivariant Chen--Ruan theories have mostly been developed for strict torus actions and symplectic reductions \cite{GoldinHolmKnutson2007,HolmMatsumura2012}, and for global quotient orbifolds \cite{FantechiGottsche2003,EdidinJarvisKimura2010}. 

This paper is organized as the following. In \S\ref{Sec_equiv_bas}, we develop the Weil model and Cartan model of equivariant basic cohomology for any Lie groupoid; in \S\ref{Sec_equiv_bss}, we use the same argument to develop the Weil model of equivariant BSS cohomology and show it consistents with the  one defined via principal \eqref{E_SloganofH}; in \S\ref{Sec_equiv_orb}, we focus on the study of the equivariant basic cohomology of orbifolds from the  differential geometric point of view, in particular, the integral localization formula is proved; in \S\ref{Sec_equiv_CR}, we equivariantize the Chen-Ruan cohomology ring.

\section{Equivariant basic cohomology of Lie groupoids}
\label{Sec_equiv_bas}

For proper Lie groupoids, Pflaum--Posthuma--Tang prove a de Rham theorem for
the complex of basic differential forms, while Watts identifies
\(\Omega^\bullet_{\mathrm{bas}}(X)\) with the de Rham complex of the orbit
space \(|X|\) equipped with its quotient diffeology; see
\cite{PflaumPosthumaTang2014,Watts2022}. 

In this section, we develop the equivariant basic cohomology for general Lie groupoids. Let $X$ be  a Lie groupoid and $G$ a Lie group. We use  Kan fibration formalism to describe a weak group $G$ action on $X$ (cf. \S\ref{Subsec_gpaction}). The key is to show that $\Omega^\bullet_\bas(X)$ is a $G^\ast$-algebra, which allows us to define the Weil/Cartan model of $H^\ast_{G,\bas}(X)$.

The organization of the section is as the following. We review the basic complex and group action on Lie groupoids in \S\ref{Subsec_bas} and \S\ref{Subsec_gpaction} respectively; in \S\ref{Subsec_equiv_bas}, we show that $\Omega_\bas(X)$ is a $G^\ast$ algebra, hence define the equivariant version of its cohomology; in \S\ref{Subsec_free_bas}, we show that if $G$ action on $X$ is locally free, $H^\ast_{G,\bas}(X)$ coincides with $H^\ast_{\bas}(X\rtimes BG).$ 

\subsection{The Basic complex of Lie groupoids}
\label{Subsec_bas}

A \emph{Lie groupoid} is a groupoid object in the category of smooth manifolds. Thus it consists of a manifold of objects \(X_0\), a manifold of arrows \(X_1\), source and target submersions
\[
  s_X,t_X:X_1\longrightarrow X_0,
\]
a unit map \(u_X:X_0\to X_1\), an inverse map \(i_X:X_1\to X_1\), and a smooth multiplication map
\[
  m_X:X_1\,{}_{s_X}\!\times_{t_X}\,X_1\longrightarrow X_1,
  \qquad (g,h)\longmapsto gh,
\]
defined for composable arrows and satisfying the usual groupoid axioms. We write
\[
  X=(X_1\rrto X_0).
\]
An arrow \(g\in X_1\) with \(s_X(g)=x\) and \(t_X(g)=y\) may be regarded as an isomorphism from \(x\) to \(y\). The orbits of \(X\) are the equivalence classes in \(X_0\) generated by the existence of arrows between points.

The Lie algebroid of \(X\) is
\[
  A=\ker(\dd s_X)|_{X_0},
  \qquad \rho=\dd t_X|_{A}:A\longrightarrow TX_0,
\]
where \(X_0\) is identified with the unit submanifold of \(X_1\). The image of the anchor \(\rho\) gives the tangent directions to the orbits.

\begin{defn}[Basic forms]
Let \(X\) be a Lie groupoid. A differential form \(\alpha\in\Om^{k}(X_0)\) is called \emph{basic} if
\(
  s_X^{*}\alpha=t_X^{*}\alpha.
\)
The space of \(k\)-forms basic for \(X\) is denoted by
\(
  \Om^{k}_{\bas}(X).
 \) 
\end{defn}

Since exterior differentiation commutes with pullback,
so \(\dd\alpha\) is basic whenever \(\alpha\) is basic. Hence the basic forms form a subcomplex of the de Rham complex of \(X_0\):
\[
  \left(\Om^{\bullet}_{\bas}(X),\dd\right)
  \subseteq
  \left(\Om^{\bullet}(X_0),\dd\right).
\]
This is called the \emph{basic deRham complex} or simply the \emph{basic complex} of the Lie groupoid \(X\). Its cohomology is denoted
\[
  H^{\bullet}_{\bas}(X)
  :=H^{\bullet}\left(\Om^{\bullet}_{\bas}(X),\dd\right).
\]

The condition \(s_X^{*}\alpha=t_X^{*}\alpha\) says that \(\alpha\) is constant along arrows of the groupoid. It simultaneously encodes two familiar requirements: invariance under the arrows and horizontality along the orbit directions.
Infinitesimally, if \(A=\Lie(X)\) and \(\rho:A\to TX_0\) is its anchor, then a basic form satisfies
\[
  \iota_{\rho(a)}\alpha=0,
  \qquad
  \mathcal{L}_{\rho(a)}\alpha=0,
  \qquad a\in\Gamma(A).
\]
Conversely, when the source fibers of \(X\) are connected, these infinitesimal conditions are equivalent to the groupoid-level condition \(s_X^{*}\alpha=t_X^{*}\alpha\). For a non-source-connected groupoid, the infinitesimal conditions must be supplemented by invariance under the remaining discrete arrows.

\begin{example}[Action groupoid]
Let a Lie group \(K\) act smoothly on a manifold \(M\). The action groupoid is
\[
  K\ltimes M\rightrightarrows M,
  \qquad
  s(k,x)=x,
  \qquad
  t(k,x)=k\cdot x.
\]
A form \(\alpha\in\Om^{k}(M)\) is basic for this groupoid exactly when
\[
  s^{*}\alpha=t^{*}\alpha
  \qquad\text{on }K\times M.
\]
This is the usual condition that \(\alpha\) be invariant under the action and horizontal along the fundamental vector fields generated by the action.
\end{example}

If the orbit space \(X_0/X\) is a smooth manifold and the quotient map
\[
  \pi:X_0\longrightarrow X_0/X
\]
is a submersion, then pullback identifies forms on the quotient with basic forms:
\[
  \pi^{*}:\Om^{\bullet}(X_0/X)\xrightarrow{\cong}\Om^{\bullet}_{\bas}(X).
\]
Thus, in the regular smooth-quotient case, the basic complex is the de Rham complex of the quotient written upstairs on \(X_0\).

It is observed that $\Omega^\ast_\bas(X)$ is a Morita invariance. 
\begin{theorem}[Morita invariance of the basic complex, Proposition 8.3, \cite{PflaumPosthumaTang2014}]
Let
\[
X=(X_1\rightrightarrows X_0),
\qquad
Y=(Y_1\rightrightarrows Y_0)
\]
be Lie groupoids, and let
\[
f:X\to Y
\]
be an equivalence. Then pullback along $f_0$ induces an isomorphism of
cochain complexes
\[
f_0^*:
\Omega^\bullet_{\mathrm{bas}}(Y)
\xrightarrow{\;\cong\;}
\Omega^\bullet_{\mathrm{bas}}(X).
\]
hence, if $f$ is a Morita equivalence, $f^\ast_0$ is still an isomorphism. 
Consequently,
\[
H^\bullet_{\mathrm{bas}}(Y)
\cong
H^\bullet_{\mathrm{bas}}(X).
\]
\end{theorem}

\subsection{Group action on Lie groupoids}\label{Subsec_gpaction}

As mentioned in the introduction, there are various versions of group action on Lie groupoids. No matter which version is used, a $G$-action on a Lie groupoid $X$, strict or weak, determines a quotient groupoid/stack $X\rtimes BG$. In the Kan-fibration model, this quotient is the total groupoid $Y$ of the fibration $Y\to BG$, whose fiber is $X$.

Let \(G\) be a Lie group. We regard \(BG\) as the one-object Lie groupoid
\[
BG=(G\rightrightarrows *),
\]
whose arrows are the elements of \(G\), with multiplication given by the group
law.

\begin{defn}[Kan fibration over \(BG\)]
A homomorphism of Lie groupoids
\[
p:Y\to BG
\]
is called a Kan fibration if the map
\[
(s_Y,p_1):Y_1\longrightarrow Y_0\times G,
\qquad
g\longmapsto (s_Y(g),p_1(g)),
\]
is a surjective submersion.
\end{defn}
Using inverses, for a Kan fibration it is equivalent to requiring that
\[
(t_Y,p_1):Y_1\longrightarrow Y_0\times G
\]
be a surjective submersion.

The fiber of \(p\) over the unique object of \(BG\) is the Lie subgroupoid
\[
X=(p_1^{-1}(e)\rightrightarrows X_0),\;\;\; \mbox{ where }\;\;\;
X_0=Y_0.
\]
It has the same object manifold \(Y_0\), and its arrows are precisely those
arrows of \(Y\) mapped to the identity element \(e\in G\). 
\begin{definition}
    A Kan fibration over \(BG\) is denoted as a diagram
\begin{equation}\label{E_KF}
X\hookrightarrow Y\xrightarrow{p} BG,
\end{equation}
we say $p$ is a (weak) $G$-action on $X$. We may formally denote $Y$ by $X\rtimes BG$.
\end{definition}
The weak \(G\)-action can be better interpreted in terms of bibundles. Given
the Kan fibration \eqref{E_KF},
each element \(a\in G\) determines the manifold
\[
P_a:=p_1^{-1}(a)\subseteq Y_1.
\]
This is naturally an \(X\)-\(X\) bibundle. Its left and right moment maps
are
\[
t_Y:P_a\to X_0,
\qquad
s_Y:P_a\to X_0.
\]
The left and right actions of \(X\) are given by multiplication in \(Y\):
\[
h\cdot g:=hg,
\qquad
g\cdot k:=gk,
\]
where
\[
g\in P_a,
\qquad
h,k\in (X)_1=p_1^{-1}(e)
\]
are composable.

Thus \(P_a\) should be interpreted as the action of \(a\) on the groupoid
\(X\). The multiplication in \(Y\) gives canonical bibundle isomorphisms
\[
P_a\otimes_{X}P_b\cong P_{ab}.
\]
Moreover,
\(
P_e=X_1
\)
is the identity bibundle of \(X\). Hence the Kan fibration encodes an action
of \(G\) on \(X\) by Morita self-equivalences, rather than by strict
automorphisms. This is why it is called a weak \(G\)-action.

\begin{defn}[\(G\)-equivariant homomorphism]
Let
\[
X\hookrightarrow Y\xrightarrow{p}BG,
\qquad
X'\hookrightarrow Y'\xrightarrow{p'}BG
\]
be Kan fibrations, regarded as weak \(G\)-actions.
A \(G\)-equivariant homomorphism
\[
f:X\longrightarrow X'
\]
is a homomorphism of Lie groupoids
$
F:Y\to Y'
$
over \(BG\), meaning that
$
p'\circ F=p,
$ and the restriction of $F$ on the fiber is $f$. This is illustrated by the diagram
\[
\begin{tikzcd}
X \arrow[r] \arrow[d,"f"']
& Y \arrow[r] \arrow[d,"F"]
& BG \arrow[d,equal] \\
 X' \arrow[r]
&  Y' \arrow[r]
& BG.
\end{tikzcd}
\]
If $f$ and $F$ are equivalence, we say two actions are equivalent.
\end{defn}

The group $G$ admits bi-actions of $G$, or equivalently, by $BG$.  Let 
$
EG=BG\ltimes G
$
be the quotient groupoid. We write $G_0$ for object space $G$. 
\[
EG=(G\times G_0\rrto G_0).
\]
Let
$
X\hookrightarrow Y\xrightarrow{p}BG
$
be a weak \(G\)-action, and let
\[
Z:=Y\times_{BG}EG.
\]
Then 
\[
Z_1=Y_1\times_G(G\times G_0)\cong Y_1\times G_0,\;\;\;
Z_0\cong Y_0\times G_0.
\]
The structure maps are 
\[
s_Z(\gamma,h)=(s_Y(\gamma),h),\;\;\;
t_Z(\gamma,h)=(t_Y(\gamma), p_1(\gamma^{-1})h),\]
\[
u_Z(x,h)=(u_Y(x),h),\;\;\;
i_Z(\gamma,h)=(i_Y(\gamma),p_1(\gamma^{-1})h ),
\]
and the multiplication is as follows. If
\[
\gamma:x\to y,
\qquad
\eta:y\to z,
\]
with
\[
p_1(\gamma)=a,
\qquad
p_1(\eta)=b,
\]
then
\[
(\gamma,h):(x,h)\to(y,a^{-1}h),
\]
and
\[
(\eta,a^{-1}h):(y,a^{-1}h)\to(z,b^{-1}a^{-1}h).
\]
Their product is
\[(\gamma,h)
(\eta,a^{-1}h)=(\gamma\eta,h).
\]

\(Z\) has a strict right \(G\)-action
\begin{equation}\label{E_strict_free}
(x,h)\cdot a=(x,ha),
\qquad
(\gamma,h)\cdot a=(\gamma,ha)
\end{equation}
which corresponds to the 
 Kan fibration
\begin{equation}\label{E_KF_strict_free}
Z\hookrightarrow \tilde Z:=Z\rtimes BG\xrightarrow{q}BG.
\end{equation}
Here 
\[
\tilde Z_0=Z_0=Y_0\times G_0,\;\;\;
\tilde Z_1=Z_1\times G
\cong Y_1\times G_0\times G.
\]
The structure maps are 
\[
\tilde s(\gamma,h,a)=(s_Y(\gamma),h),\;\;
\tilde t(\gamma, h,a)=(t_Y(\gamma), ha),\]
\[\tilde u(x,h)=(u_X(x), h,e),\;\; \tilde i(\gamma,h,a)=(i_X(\gamma),p_1^{-1}(\gamma)ha, a^{-1}).
\]
$q_1(\gamma,h,a)=a$.

\begin{theorem} The diagram 
\[
\begin{tikzcd}
X \arrow[r] \arrow[d,"f"']
& Y \arrow[r] \arrow[d,"F"]
& BG \arrow[d,equal] \\
 Z \arrow[r]
&  Z\rtimes BG \arrow[r]
& BG
\end{tikzcd}
\]
gives a  $G$-equivariant homomorphism between two Kan fibrations, where 
$
F: Y\to Z\rtimes BG $ is given by 
\[
F_0(x)=(x,e),\;\;\;
F_1(\gamma)=(\gamma, p_1(\gamma),e).
\]
Moreover, both $f$ and $F$ are equivalence. Hence, two actions are equivalent.
\end{theorem}
The proof is straightforward. 

Note that the $G$ action on $Z$ is strict and free on each level (cf. \eqref{E_strict_free}).
This theorem says that we can replace $X$ by an equivalence $Z$ such that the equivalent $G$ action on $Z$ is strict and free on both $Z_0$ and $Z_1$. In \cite{BehrendXu2011,BursztynNosedaZhu2020}, such   $Z$ is also called  a $G$-principal bundle over 
\[
Z/G:=(Z_1/G\rrto Z_0/G).
\]
It is clear that
\begin{prop}\label{P_ZtoY}
    $Z/G\cong Y$.
\end{prop}
The groupoid $Z$ and the Kan fibration 
\eqref{E_KF_strict_free}
play important roles in this section and the next section.
\subsection{$G$ equivariant basic cohomology}\label{Subsec_equiv_bas}

Let \(Z\) be the Lie groupoid  carrying a strict  \(G\)-action.

\begin{prop}
\(\Omega^\bullet_{\mathrm{bas}}(Z)\) is naturally a \(G^\ast\)-algebra.
\end{prop}
\begin{proof}
For \(\xi\in\mathfrak g=\Lie(G)\), let \(\xi_Z\) be the fundamental vector field
on \(Z_0\). Define
\[
\iota_\xi:=\iota_{\xi_Z},
\qquad
\mathcal L_\xi:=\mathcal L_{\xi_Z}
\]
on \(\Omega^\bullet_{\mathrm{bas}}(Z)\). These operations preserve basic forms
and satisfy the Cartan relations
\[
[\dd,\iota_\xi]=\mathcal L_\xi,
\qquad
[\dd,\mathcal L_\xi]=0,
\qquad
[\mathcal L_\xi,\iota_\eta]=\iota_{[\xi,\eta]},
\qquad
[\mathcal L_\xi,\mathcal L_\eta]=\mathcal L_{[\xi,\eta]}.
\]
Thus \(\Omega^\bullet_{\mathrm{bas}}(Z)\) is a differential graded algebra
equipped with the standard \(G^\ast\)-algebra structure.
\end{proof}

Let
\[
A^\bullet:=\Omega^\bullet_{\bas}(Z).
\]
Since \(A^\bullet\) is a \(G^\ast\)-algebra, its equivariant cohomology can be
defined by the Weil model or, equivalently, by the Cartan model. This is explained in \cite{GuilleminSternberg1999}.

The Weil algebra of \(\mathfrak g=\Lie(G)\) is
\[
W(\mathfrak g)=S(\mathfrak g^\vee)\otimes \wedge(\mathfrak g^\vee),
\]
where elements of \(\wedge^1(\mathfrak g^\vee)\) have degree \(1\), and
elements of \(S^1(\mathfrak g^\vee)\) have degree \(2\). Here $\mathfrak g^\vee$ is the dual space of $\mathfrak g$

The \emph{Weil model} of the equivariant cohomology of the basic complex is
\[
\left(W(\mathfrak g)\otimes \Omega^\bullet_{\bas}(Z)\right)_{\bas},
\]
where ``basic'' means horizontal and invariant for the diagonal
\(G^\ast\)-structure. Its differential is
\[
d_W\otimes 1+1\otimes d.
\]
Thus
\[
H_G^\bullet\bigl(\Omega^\bullet_{\bas}(Z)\bigr)
:=
H^\bullet\left(
\left(W(\mathfrak g)\otimes \Omega^\bullet_{\bas}(Z)\right)_{\bas},
d_W\otimes 1+1\otimes d
\right).
\]

Equivalently, the \emph{Cartan model} is
\[
\left(S(\mathfrak g^\vee)\otimes \Omega^\bullet_{\bas}(Z)\right)^G
\]
with differential
\[
d_C
=
1\otimes d-\sum_a u^a\otimes \iota_{e_a},
\]
where \(\{e_a\}\) is a basis of \(\mathfrak g\), and \(\{u^a\}\) is the
corresponding degree \(2\) basis of \(S^1(\mathfrak g^\vee)\). Hence
\[
H_G^\bullet\bigl(\Omega^\bullet_{\bas}(Z)\bigr)
\cong
H^\bullet\left(
\left(S(\mathfrak g^\vee)\otimes \Omega^\bullet_{\bas}(Z)\right)^G,
d_C
\right).
\]

\begin{defn}[\(G\)-equivariant basic cohomology]
Let
\[
X\hookrightarrow Y\xrightarrow{p}BG
\]
be a weak \(G\)-action, and define
\[
Z:=Y\times_{BG}EG.
\]
Then \(Z\) is Morita equivalent to \(X\), and \(Z\) carries a strict free
\(G\)-action. Hence
\[
\Omega^\bullet_{\bas}(X)\cong \Omega^\bullet_{\bas}(Z)
\]
is a \(G^\ast\)-algebra.

We define the \(G\)-equivariant basic cohomology of \(X\) by
\[
H^\bullet_{G,\bas}(X)
:=
H^\bullet_G\bigl(\Omega^\bullet_{\bas}(X)\bigr).
\]
Equivalently, in the Weil model,
\[
H^\bullet_{G,\bas}(X)
:=
H^\bullet\left(
\left(W(\mathfrak g)\otimes \Omega^\bullet_{\bas}(X)\right)_{\bas},
d_W\otimes 1+1\otimes d
\right).
\]

In the Cartan model, one writes
\[
H^\bullet_{G,\bas}(X)
\cong
H^\bullet\left(
\left(S(\mathfrak g^\vee)\otimes \Omega^\bullet_{\bas}(Z)\right)^G,
d_C
\right),
\]
where
\[
d_C
=
1\otimes d-\sum_a u^a\otimes\iota_{e_a}.
\]
Here \(\{e_a\}\) is a basis of \(\mathfrak g\), and \(\{u^a\}\) is the
corresponding degree \(2\) basis of \(S^1(\mathfrak g^\vee)\).
\end{defn}

\subsection{Equivariant basic cohomology for locally free actions}\label{Subsec_free_bas}

It is well known that if $G$ acts freely on a smooth manifold $M$, we have 
\[
H^\ast_G(M)\cong H^\ast(M/G).
\]
We show that this is still true for $H^\ast_{G,\bas}$. 

There is a issue of what we mean a free $G$ action on a Lie groupoid.
Given a Kan fibration \eqref{E_KF}, it induces a $G$ action on the coarse space $|X|$.  

\begin{defn}[Locally free \(G\)-action] 
We say the $G$ action given by the Kan fibration 
\eqref{E_KF} is (locally) free if  the induced
action on \(|X|\) is (locally) free. Equivalently, for every point
\[
[x]\in |X|,
\]
the stabilizer subgroup
\[
G_{[x]}:=\{g\in G\mid g\cdot [x]=[x]\}
\]
is (discrete, or) trivial.
\end{defn}

For \(x\in X_0\), let
\[
{\Gamma}_x:=\{\gamma\in X_1\mid s_X(\gamma)=t_X(\gamma)=x\}
\]
be the isotropy group of \(X\) at \(x\), and
\[
\widehat{\Gamma}_x:=\{\gamma\in Y_1\mid s_Y(\gamma)=t_Y(\gamma)=x\}
\]
be the isotropy group of \(Y\) at \(x\). Then there is an exact sequence of groups
\begin{equation}\label{E_isotropy}
1\to \Gamma_x\to \widehat{\Gamma}_x\xrightarrow{p_1} G.
\end{equation}
One has
\[
G_{[x]}\cong p_1(\widehat{\Gamma_x}).
\]
Hence the weak \(G\)-action is locally free exactly when
\(
p_1(
\widehat{\Gamma_x})\subseteq G
\)
is discrete for every \(x\in X_0\).  
\begin{prop}\label{P_Z_C}
Let $Z$
be a proper Lie groupoid with compact orbit space \(|Z|\). Let \(G\) be a
compact Lie group acting strictly on \(Z\) and freely on both $Z_0$ and $Z_1$. Assume that the induced \(G\)-action
on \(|Z|\) is locally free. Then
\[
\Omega^\bullet_{\bas}(Z)
\]
satisfies condition \((C)\), i.e, there exists a connection form
\[
\theta\in \Omega^1_{\bas}(Z)\otimes\mathfrak g
\]
such that, for every \(\xi\in\mathfrak g\),
\[
\iota_{\xi_Z}\theta=\xi,
\]
and, for the right \(G\)-action,
\[
R_a^*\theta=\operatorname{Ad}_{a^{-1}}\theta.
\]
\end{prop}

\begin{proof}
Since \(Z\) is proper and \(G\) is compact, \(Q=Z\rtimes BG\) is a proper Lie groupoid.
By the slice theorem for proper Lie groupoids, \(Q\) admits slices at every
point of \(Z_0\). Since $|Q|$ is compact, therefore one may choose finitely many slices and a
\(Q\)-invariant partition of unity subordinate to the corresponding slice
neighborhoods.

Let
\[
A_Z=\Lie(Z),
\qquad
\rho_Z:A_Z\to TZ_0,
\]
and write
\[
\nu_Z:=TZ_0/\rho_Z(A_Z)
\]
for the normal directions to the \(Z\)-orbits. Since the \(G\)-action on
\(|Z|\) is locally free, the map
\[
\mathfrak g\longrightarrow \nu_{Z,z},
\qquad
\xi\longmapsto \overline{\xi_Z(z)}
\]
is injective for every \(z\in Z_0\).

On each slice chart, choose a splitting of the normal directions
\[
\nu_Z
=
\overline{\mathfrak g_Z}\oplus H_i,
\]
where
\[
\overline{\mathfrak g_Z}
=
\{\overline{\xi_Z}\mid \xi\in\mathfrak g\}.
\]
This splitting defines a local connection form
\[
\theta_i\in \Omega^1_{\bas}(Z)\otimes\mathfrak g
\]
by requiring
\[
\theta_i(\xi_Z)=\xi,
\qquad
\theta_i|_{\rho_Z(A_Z)}=0,
\qquad
\theta_i|_{H_i}=0.
\]
Equivalently,
\[
\iota_{\xi_Z}\theta_i=\xi.
\]
The slice construction is \(Q\)-invariant, hence the local forms may be chosen
\(G\)-equivariant:
\[
R_a^*\theta_i=\operatorname{Ad}_{a^{-1}}\theta_i.
\]

Let \(\{\varphi_i\}\) be a \(Q\)-invariant partition of unity subordinate to
the slice cover. Define
\[
\theta:=\sum_i \varphi_i\theta_i.
\]
Since the functions \(\varphi_i\) are \(Q\)-invariant, hence \(Z\)-basic and
\(G\)-invariant, the form \(\theta\) is still basic:
\[
\theta\in\Omega^1_{\bas}(Z)\otimes\mathfrak g.
\]
Moreover,
\[
\iota_{\xi_Z}\theta
=
\sum_i \varphi_i\,\iota_{\xi_Z}\theta_i
=
\sum_i \varphi_i\,\xi
=
\xi,
\]
and
\[
R_a^*\theta=\operatorname{Ad}_{a^{-1}}\theta.
\]
Thus \(\theta\) is a connection form in the \(G^\ast\)-algebra
\(
\Omega^\bullet_{\bas}(Z).
\)
Therefore \(\Omega^\bullet_{\bas}(Z)\) satisfies condition \((C)\).
\end{proof}

\begin{theorem}[\cite{GuilleminSternberg1999}]\label{T_C}
Let \(A\) be a \(G^\ast\)-algebra. Define its \(G\)-basic subcomplex by
\[
A_{\bas}
=
\{\alpha\in A\mid \iota_\xi\alpha=0,\ \mathcal L_\xi\alpha=0
\text{ for all }\xi\in\mathfrak g\}.
\]
If \(A\) satisfies condition \((C)\), namely if \(A\) admits a connection form,
then the natural inclusion of \(G\)-basic elements into the Weil model induces
an isomorphism
\[
H^\bullet(A_{\bas})
\cong
H_G^\bullet(A).
\]
\end{theorem}

We apply this theorem to the case
\[
A^\bullet:=\Omega^\bullet_{\bas}(Z),
\qquad
Z=Y\times_{BG}EG.
\]
Then $\pi: Z\to Z/G=Y$ is
\[
\pi_0:Y_0\times G_0\to Y_0,
\qquad
\pi_0(x,h)=x,
\]
and
\[
\pi_1:Y_1\times G_0\to Y_1,
\qquad
\pi_1(\gamma,h)=\gamma.
\]
\begin{prop}\label{P_ZandY}
Pullback gives an identification
\[
\pi_0^*:\Omega^\bullet_{\bas}(Y)\xrightarrow{\cong}
\left(\Omega^\bullet_{\bas}(Z)\right)_{\bas}.
\]
\end{prop}
\begin{proof}
If \(\alpha\in\Omega^\bullet_{\bas}(Y)\), then
\[
s_Z^*\pi_0^*\alpha
=
\pi_1^*s_Y^*\alpha
=
\pi_1^*t_Y^*\alpha
=
t_Z^*\pi_0^*\alpha,
\]
so
\[
\pi_0^*\alpha\in\Omega^\bullet_{\bas}(Z).
\]
Moreover, \(\pi_0^*\alpha\) is horizontal and invariant for the strict
\(G\)-action, hence
\[
\pi_0^*\alpha\in
\left(\Omega^\bullet_{\bas}(Z)\right)_{\bas}.
\]
Conversely, let
\(
\beta\in\left(\Omega^\bullet_{\bas}(Z)\right)_{\bas}.
\)
Since the \(G\)-action on
\(
Z_0
\)
is free and \(\beta\) is \(G\)-basic, there is a unique form
\(
\alpha\in\Omega^\bullet(Y_0)
\)
such that
\(
\beta=\pi_0^*\alpha.
\)
Since \(\beta\) is \(Z\)-basic,
\(
s_Z^*\beta=t_Z^*\beta.
\)
Therefore
\[
\pi_1^*s_Y^*\alpha=\pi_1^*t_Y^*\alpha.
\]
Because
\[
\pi_1:Y_1\times G_0\to Y_1
\]
is a surjective submersion, pullback by \(\pi_1\) is injective. Hence
\[
s_Y^*\alpha=t_Y^*\alpha.
\]
Thus
\(
\alpha\in\Omega^\bullet_{\bas}(Y).
\)
Therefore
\[
A_{\bas}
=
\left(\Omega^\bullet_{\bas}(Z)\right)_{\bas}
\cong
\Omega^\bullet_{\bas}(Y).
\]
\end{proof}

We prove the following theorem.
\begin{theorem}
    Let $X\hookleftarrow Y\xrightarrow{p} BG$ be a Kan fibration. 
  Suppose $X$ is proper and compact, $G$ is compact.  If the $G$ action on $X$ is locally free, then 
  \[
  H^\ast_{G,\bas} (X)\cong 
  H^\ast_\bas (Y).
  \]
\end{theorem}
\begin{proof}
   Since  $X$ is compact and proper, so is $Z$. Then 
   \[
  H^\ast_{G,\bas} (X)\cong
  H^\ast_{G,\bas} (Z)\cong
  H^\ast ((\Omega^\bullet_\bas(Z))_\bas,d)\cong H^\ast(\Omega^\bullet_\bas(Y),d)=
  H^\ast_\bas (Y).
  \] 
  We use Proposition \ref{P_Z_C} and Theorem \ref{T_C} for the second
  isomorphism, Proposition \ref{P_ZandY} for the third isomorphism.
 \end{proof}

\section{Equivariant Bott--Shulman--Stasheff cohomology}\label{Sec_equi_BSS}\label{Sec_equiv_bss}

The Bott--Shulman--Stasheff complex is the simplicial de Rham complex of the nerve of a Lie groupoid~\cite{BottShulmanStasheff1976}. Its cohomology is Morita invariant and is commonly regarded as the de Rham cohomology of the differentiable stack presented by the groupoid; see Behrend's~\cite{Behrend2005} construction of de Rham cohomology of differentiable stacks via a groupoid double complex, and Behrend--Xu~\cite{BehrendXu2011} for Morita invariance of such stack cohomology. Usually, 
$H_{G,\BSS}^\ast(X)$ is defined to be $H^\ast_\BSS(X\rtimes BG)$ (cf. Definition \ref{D_equiv_BSS}). In this section, we replace $X$ by its equivalence $Z$ and show that $H^\ast_{G,\BSS}(X)$ can be accomplished by a Weil model on $\Omega^\bullet_\BSS(Z)$ (cf. Theorem \ref{T_BSS}).

For any Lie groupoid
\[
Y=(Y_1\rrto Y_0),
\]
let \(Y_\bullet\) be its nerve. Thus \(Y_q\) is the manifold of composable
\(q\)-tuples of arrows. The Bott--Shulman--Stasheff double complex is
\[
\Omega^{p,q}_{\mathrm{BSS}}(Y):=\Omega^p(Y_q),
\]
with de Rham differential \(d\) and simplicial differential
\[
\delta=\sum_i(-1)^i d_i^*.
\]
The total complex is denoted
\[
(\Omega^\bullet_{\mathrm{BSS}}(Y),D),
\]
where
\[
\Omega^n_{\mathrm{BSS}}(Y)
=
\bigoplus_{p+q=n}\Omega^p(Y_q),
\qquad
D=d+(-1)^p\delta
\]
on \(\Omega^p(Y_q)\). Its cohomology is
\[
H^\bullet_{\mathrm{BSS}}(Y)
:=
H^\bullet(\Omega^\bullet_{\mathrm{BSS}}(Y),D).
\]

\begin{defn}[\(G\)-equivariant BSS cohomology]\label{D_equi_BSS}\label{D_equiv_BSS}
Let
\[
X\hookrightarrow Y\xrightarrow{p}BG
\]
be a Kan fibration.
The \(G\)-equivariant Bott--Shulman--Stasheff cohomology of \(X\) is defined by
\[
H^\bullet_{G,\mathrm{BSS}}(X)
:=
H^\bullet_{\mathrm{BSS}}(Y).
\]
Equivalently,
\[
H^\bullet_{G,\mathrm{BSS}}(X)
=
H^\bullet\bigl(\Omega^\bullet_{\mathrm{BSS}}(Y),D\bigr).
\]
\end{defn}

\begin{prop}
Let
\(
Z
\)
be a Lie groupoid with a strict right \(G\)-action. Assume that \(G\) acts freely
on both \(Z_0\) and \(Z_1\). Then
\[
\Omega^\bullet_{\BSS}(Z)
\]
is a \(G^\ast\)-algebra. Moreover, it satisfies condition \((C)\).
\end{prop}

\begin{proof}
Since the \(G\)-action on \(Z\) is strict and free, it acts on every nerve space
\[
Z_q=Z_1\times_{Z_0}\cdots\times_{Z_0}Z_1
\]
by the diagonal action
\[
(\gamma_1,\ldots,\gamma_q)\cdot a
=
(\gamma_1\cdot a,\ldots,\gamma_q\cdot a).
\]
For \(q=0\), this is the given action on \(Z_0\). Hence \(G\) acts freely on
every \(Z_q\).
All face and degeneracy maps of the nerve are \(G\)-equivariant. Therefore the
pullback action, contraction, and Lie derivative define operators on each
\(
\Omega^p(Z_q).
\)
For \(\xi\in\mathfrak g\), let \(\xi_q^\#\) be the fundamental vector field on
\(Z_q\). Define
\[
\iota_\xi|_{\Omega^p(Z_q)}:=\iota_{\xi_q^\#},
\qquad
\mathcal L_\xi|_{\Omega^p(Z_q)}:=\mathcal L_{\xi_q^\#}.
\]
Since the face maps are \(G\)-equivariant, contraction and Lie derivative
commute with the simplicial differential \(\delta\). Thus, for the BSS total
differential
\[
D=d+(-1)^p\delta
\qquad\text{on }\Omega^p(Z_q),
\]
one has
\[
[D,\iota_\xi]=\mathcal L_\xi,
\qquad
[D,\mathcal L_\xi]=0.
\]
Together with the usual Cartan identities on each \(Z_q\), this makes
\(
\Omega^\bullet_{\BSS}(Z)
\)
a \(G^\ast\)-algebra.

Now, since \(G\) acts freely on \(Z_0\), the projection
\[
Z_0\to Z_0/G
\]
is a principal \(G\)-bundle. Choose a principal connection form
\(
\theta\in\Omega^1(Z_0)\otimes\mathfrak g.
\)
Write
\[
\theta=\sum_a \theta^a\otimes \xi_a
\]
with respect to a basis \(\{\xi_a\}\) of \(\mathfrak g\). Since
\[
\Omega^1(Z_0)\subset \Omega^1_{\BSS}(Z)
\]
as the component of bidegree \((1,0)\), each \(\theta^a\) is an element of total
degree \(1\) in \(\Omega^\bullet_{\BSS}(Z)\).
Therefore \(\Omega^\bullet_{\BSS}(Z)\) satisfies condition \((C)\).
\end{proof}

\begin{remark}
One may use $\delta+(-1)^qd$
as the differential on $\Omega_\BSS^\bullet$. However, such a definition is not compatible with $G^\ast$ structure.
\end{remark}

\begin{prop}
Let
\[
X\hookrightarrow Y\xrightarrow{p}BG
\]
be a Kan fibration, and set
\[
Z:=Y\times_{BG}EG.
\]
Then
\[
\bigl(\Omega^\bullet_{\BSS}(Z)\bigr)_{\bas}
\cong
\Omega^\bullet_{\BSS}(Y)
\]
as cochain complexes. 
\end{prop}

\begin{proof}
The strict \(G\)-action on \(Z\) is free and, by Proposition \ref{P_ZtoY},
\(
Z/G\cong Y.
\)
For each \(q\geq 0\), 
\[
\pi_q:Z_q\to Y_q
\]
is a principal \(G\)-bundle, pullback
identifies differential forms on \(Y_q\) with \(G\)-basic forms on \(Z_q\):
\[
\pi_q^*:\Omega^p(Y_q)\xrightarrow{\cong}\Omega^p(Z_q)_{\bas}.
\]
Therefore, for every total degree \(n\),
\[
\Omega^n_{\BSS}(Y)
=
\bigoplus_{p+q=n}\Omega^p(Y_q)
\cong
\bigoplus_{p+q=n}\Omega^p(Z_q)_{\bas}
=
\bigl(\Omega^n_{\BSS}(Z)\bigr)_{\bas}.
\]
 Therefore it is an
isomorphism of cochain complexes.
\end{proof}

Consequently, since \(\Omega^\bullet_{\BSS}(Z)\) is a \(G^\ast\)-algebra of
condition \((C)\), we have
\[
H_G^\bullet\bigl(\Omega_{\BSS}(Z)\bigr)
\cong
H^\bullet\bigl((\Omega_{\BSS}(Z))_{\bas}\bigr).
\]
Using the proposition,
\[
H^\bullet\bigl((\Omega_{\BSS}(Z))_{\bas}\bigr)
\cong
H^\bullet(\Omega_{\BSS}(Y))
=
H^\bullet_{\BSS}(Y).
\]
By definition,
\[
H^\bullet_{G,\BSS}(X)
:=
H^\bullet_{\BSS}(Y).
\]
Thus
\begin{theorem}\label{T_BSS}
$ H^\bullet_{G,\BSS}(X)
\cong
H_G^\bullet\bigl(\Omega_{\BSS}(Z),D\bigr),
$ where 
\[
H^\bullet_{G}(\Omega_\BSS(Z),D)
:=
H^\bullet\left(
\left(W(\mathfrak g)\otimes \Omega^\bullet_{\BSS}(Z)\right)_{\bas},
d_W\otimes 1+1\otimes D
\right).
\]
\end{theorem}

\section{Equivariant cohomology on orbifolds with weak group actions}\label{Sec_equiv_orb}

An orbifold is known as a proper \'etale Lie groupoid.
For a smooth manifold, differential-geometric objects are usually described by bundles and tensor fields on the manifold: differential forms, vector fields, Riemannian metrics, almost complex structures, symplectic forms, characteristic forms, and so on.
For an orbifold presented by a proper \'{e}tale Lie groupoid,
the same objects are described by $X$-equivariant geometric data on $X_
0$.

Thus differential geometry on manifolds extends to orbifolds by replacing ordinary tensor fields on manifolds with invariant, or basic, tensor fields on a proper \'{e}tale groupoid presentation.
This is why equivariantizing the basic complex is geometrically meaningful: it gives the equivariant version of the ordinary differential-geometric de Rham theory of an orbifold.

    \subsection{Differential geometry on orbifolds}
Let
\[
        X=(X_1\rightrightarrows X_0)
\]
be a proper \'{e}tale Lie groupoid, i.e, an orbifold. 
Define 
\[
TX=(TX_1\rrto TX_0)
\]
whose source and target maps are $ds$ and $dt$. 
Since \(X\) is \'{e}tale, the source and
target maps
\[
        s,t:X_1\to X_0
\]
are local diffeomorphisms.  Hence every arrow
\(
        g:x\to y
\)
of \(X\) determines a germ of a local diffeomorphism from a neighborhood of
\(x\) to a neighborhood of \(y\).  Taking differentials gives a canonical
linear isomorphism
\[
        dg:T_xX_0\longrightarrow T_yX_0 .
\]
Thus \(X\) acts canonically on the tangent bundle
\(TX_0\).  We know    
\[
        TX=TX_0\rtimes X .
\]

Once \(TX\) is defined, the usual tensor constructions give well-defined tensor
bundles on the orbifold:
\[
        T^*X,\qquad
        \Lambda^kT^*X,\qquad
        \operatorname{Sym}^kT^*X,\qquad
        TX\otimes T^*X,
\]
and so on.  Concretely, these are represented by the corresponding tensor
bundles on \(X_0\), equipped with the induced \(X\)-action.

Therefore tensor fields on the orbifold are precisely \(X\)-invariant sections
of these equivariant tensor bundles.  For example, a differential \(k\)-form on
the orbifold is an \(X\)-invariant section of \(\Lambda^kT^*X\).  In terms of
the groupoid presentation, this is exactly a form
\(
        \alpha\in \Omega^k(X_0)
\)
satisfying
\(
        s^*\alpha=t^*\alpha .
\)
Thus
\[
        \Omega^k_{\bas}(X)
        =
        \Gamma(\Lambda^kT^*X)
        =
        \left\{
        \alpha\in\Omega^k(X_0)
        \ \middle|\
        s^*\alpha=t^*\alpha
        \right\}.
\]

This explains why the complex of basic forms is the natural de Rham complex for
orbifold differential geometry.  It is the direct analogue of
\(\Omega^\bullet(M)\) for a smooth manifold \(M\).  In particular, geometric
structures such as Riemannian metrics, almost complex structures, symplectic
forms, characteristic forms, and volume forms are described by the corresponding
\(X\)-invariant tensor fields on \(X_0\).

Given a Kan fibration \eqref{E_KF} where $X$ is an orbifold.
For \(i=0,1\), consider the relative tangent bundle of
\(p_i\):
\[
        T_pY_i:=\ker(dp_i)\subset TY_i .
\]
Since \((BG)_0=*\), we have
\[
        T_pY_0=\ker(dp_0)=TY_0 .
\]
Since \(p:Y\to BG\) is a Kan fibration, the map \(p_1:Y_1\to G\) is a
submersion, so
\[
        T_pY_1=\ker(dp_1)\subset TY_1
\]
is a smooth vector subbundle.

The structure maps of \(Y\) restrict to the relative tangent bundles.  Indeed,
if
\[
        v_\gamma\in \ker(dp_1)\subset T_\gamma Y_1,
\]
then
\[
        ds_Y(v_\gamma),\,dt_Y(v_\gamma)\in TY_0=T_pY_0 .
\]
Moreover, because \(p_1\) is multiplicative, the differential of multiplication
sends composable relative tangent vectors to relative tangent vectors.  Thus we
obtain a Lie groupoid
\[
        T_pY:=(T_pY_1\rightrightarrows T_pY_0).
\]
There is a natural Lie groupoid homomorphism
\[
        p^T:T_pY\longrightarrow BG
\]
defined on objects by the unique map
\[
        T_pY_0=TY_0\longrightarrow *,
\]
and on arrows by
\[
        p^T_1(v_\gamma):=p_1(\gamma),
        \qquad
        v_\gamma\in T_\gamma Y_1,\quad dp_1(v_\gamma)=0.
\]

\begin{prop}
The homomorphism
\[
        p^T:T_pY\longrightarrow BG
\]
is a Kan fibration.  Its fiber is the tangent groupoid of \(X\):
\[
        \ker(p^T)=TX.
\]
Thus \(p^T:T_pY\to BG\) represents the weak \(G\)-action induced on the tangent
bundle of \(X\).
\end{prop}

\begin{proof}
For Lie groupoids, the Kan fibration condition over \(BG\) is equivalent to the
condition that
\[
        (s_Y,p_1):Y_1\longrightarrow Y_0\times G
\]
is a surjective submersion.  The corresponding map for \(p^T\) is
\[
        (s_{T_pY},p^T_1):T_pY_1\longrightarrow T_pY_0\times G,
\]
that is,
\[
        (s_{T_pY},p^T_1):
        \ker(dp_1)\longrightarrow TY_0\times G,
        \qquad
        v_\gamma\longmapsto \bigl(ds_Y(v_\gamma),p_1(\gamma)\bigr).
\]

This map is the relative tangent map of the surjective submersion
\[
        (s_Y,p_1):Y_1\to Y_0\times G
\]
with respect to the projection to \(G\).  Equivalently, for each \(g\in G\), the
restriction
\[
        s_Y:p_1^{-1}(g)\longrightarrow Y_0
\]
is a surjective submersion, and
\[
        (s_{T_pY},p^T_1)
\]
is obtained by taking the tangent maps of these submersions fiberwise over
\(g\).  Hence
\[
        (s_{T_pY},p^T_1):T_pY_1\to TY_0\times G
\]
is again a surjective submersion.  Therefore \(p^T:T_pY\to BG\) is a Kan
fibration.

It remains to identify its fiber.  The object space of the fiber is
\[
        T_pY_0=TY_0=TX_0.
\]
The arrow space of the fiber is
\[
        (p^T_1)^{-1}(e)
        =
        \{v_\gamma\in \ker(dp_1)\mid p_1(\gamma)=e\}.
\]
Since
\[
        X_1=p_1^{-1}(e)
\]
and \(p_1\) is a submersion, we have
\[
        T_\gamma X_1=\ker(dp_1)_\gamma
        \qquad
        \text{for }\gamma\in X_1.
\]
Therefore
\[
        (p^T_1)^{-1}(e)=TX_1.
\]
The source, target, unit, inverse, and multiplication maps are precisely the
differentials of the corresponding structure maps of \(X\).  Hence
\[
        \ker(p^T)=(TX_1\rightrightarrows TX_0)=TX.
\]
This proves the claim.
\end{proof}

If \(X\) is a proper \'{e}tale Lie groupoid, then \(TX\) is canonically the
tangent orbibundle of the orbifold represented by \(X\).  Since \(X\) is
\'{e}tale, the tangent groupoid \(TX=(TX_1\rightrightarrows TX_0)\) may also be
identified with the vector-bundle action groupoid
\[
        TX_0\rtimes X.
\]
Thus the Kan fibration
\[
        p^T:T_pY\to BG
\]
is the Kan-fibration model for the weak \(G\)-action on the tangent bundle of
the orbifold \(X\).

The same construction applies to tensor bundles.  Applying tensor operations to
the relative tangent bundle gives weak \(G\)-actions on
\[
        T^*X,\qquad
        \Lambda^kT^*X,\qquad
        \operatorname{Sym}^kT^*X,\qquad
        TX\otimes T^*X,
\]
and more generally on any tensor bundle functorially constructed from \(TX\)
and \(T^*X\).

 Suppose that a vector bundle on the quotient is given at the
Kan-fibration level by a vector bundle
\[
        \Pi:\mathcal E\longrightarrow Y
\]
whose restriction to the fiber is a vector bundle
\[
        E\longrightarrow X .
\]
Equivalently, one has a pullback square
\[
\begin{tikzcd}
        E \arrow[r] \arrow[d]
        &
        \mathcal E \arrow[d,"\Pi"]
        \\
        X \arrow[r,hook,"i"]
        &
        Y .
\end{tikzcd}
\]
Together with the map to \(BG\), this sits in a diagram
\begin{equation}\label{E_equi_VB}
\begin{tikzcd}
E \arrow[r,hook,"j_E"] \arrow[d,"\pi_E"']
&
\mathcal E \arrow[r,"p_{\mathcal E}"] \arrow[d,"\Pi"]
&
BG \arrow[d,equal]
\\
X \arrow[r,hook,"i"']
&
Y \arrow[r,"p"']
&
BG .
\end{tikzcd}
\end{equation}

A \(G\)-invariant
section of \(E\to X\) is a section
\(
        s:X\to E
\)
which extends to a section
\(
        \widetilde s:Y\to \mathcal E
\)
of \(\Pi:\mathcal E\to Y\) that compatible with the above diagram. 
We call such a section \(s\) an equivariant tensor field on \(X\).  Namely, we make the following definition.
\begin{defn}
A $G$ invariant tensor field on \(X\) is a section of the tensor bundle
\(E\to X\) which extends to a section of the corresponding bundle
\(\mathcal E\to Y\) over the Kan fibration \(Y\to BG\).
\end{defn}
For tensor bundles \(E=T^{r,s}X\), \(\Lambda^kT^*X\), \(\mathrm{Sym}^kT^*X\), and so
on, this gives the corresponding notions of $G$-invariant vector fields,
 differential forms,  metrics, almost complex
structures, etc.

The diagram \eqref{E_equi_VB} can be used to define $G$-equivariant vector bundle $E$ over $X$ and $G$-invariant sections.
\def \Vect{\mathrm{Vect}}

\begin{prop}[Existence of a \(G\)-invariant orbifold metric]\label{P_inv_metric}
Given a Kan fibration \eqref{E_KF}
with \(X\) being a proper \'{e}tale Lie groupoid, hence an
orbifold groupoid.  Assume also that \(G\) is compact.
Then \(X\) admits a \(G\)-invariant Riemannian metric. 
\end{prop}

\begin{proof} Consider $E=\mathrm{Sym}^2T^*X$ is a tensor vector bundle over $X$. Then we have 
\[
ds_Y: (E\rtimes BG)_1\to E_0,\;\;\;
dt_Y: (E\rtimes BG)_1\to E_0.
\]
They are bundle maps that covers 
\[
s_Y: Y_1\to X_0,\;\;\;
t_Y:Y_1\to X_0.
\]
In particular, $ds_Y$ (and $dt_Y$) are fiberwise isomorphisms (cf. \eqref{E_ds_etale} and  \eqref{E_dt_etale}). 

Since \(X\) is a proper \'{e}tale groupoid, it admits an orbifold Riemannian
metric.  Choose such a metric
\[
        h\in \Gamma(\mathrm{Sym}^2T^*X).
\]
For each \(g\in G\), define a new tensor field \(g^*h\) on \(X_0\) by $ds_Y\circ dt_Y^{-1}$.
Be precise, given \(x\in X_0\), choose an arrow
\[
        \gamma:x\to y
\]
in \(Y_g\).  Define
\[
        (g^*h)(x)
        =
        ds_{Y,\gamma}\circ dt^{-1}_{Y,\gamma} (h(y)),
\]
this is independent of the choice of \(\gamma\)  since $h$ is $X_1$ invariant. 
\(g^*h\) is another orbifold Riemannian metric on \(X\).

Let \(dg\) be the normalized Haar measure on the compact Lie group \(G\).
Define
\[
        \overline h
        :=
        \int_G g^*h\,dg .
\]
Explicitly,
\[
        \overline h_x(v,w)
        =
        \int_G (g^*h)_x(v,w)\,dg .
\]
Since \(G\) is compact, this integral defines a metric.  

It remains to prove that \(\overline h\) is invariant. This is to extend $\overline h$ to a section on $Y$. The construction is routine, we skip the details.
\end{proof}

We explain the infinitesimal action of a $G$ action given by the Kan fibration \eqref{E_KF}. This expects to be a map from $\mathfrak g$ to $\Vect(X)$ and is $G$ equivariant. 

We  interprete an infinitesimal action as a section of the bundle 
\[
\pi:\mathfrak g\times TX\to \mathfrak g\times X
\]
and the section is $G$ equivariant. The bundle being $G$ equivariant, we associate it with a diagram of Kan fibrations 
\begin{equation}
\begin{tikzcd}
        TX\times \mathfrak g
            \arrow[r,hook]
            \arrow[d]
        &
T_pY\times_{BG}        \mathfrak g_{BG}
            \arrow[r]
            \arrow[d]
        &
        BG
            \arrow[d,equal]
        \\
X\times        \mathfrak g
            \arrow[r,hook]
        &
Y\times_{BG}        \mathfrak g_{BG}
            \arrow[r]
        &
        BG .
\end{tikzcd}
\end{equation}
Here
the adjoint action of $G$ on its Lie algebra $\mathfrak g$ yields the vector bundle  
\[
        \mathfrak g_{BG}:=\mathfrak g\rtimes G\longrightarrow BG,
\]
which is also a Kan fibration. The infinitesimal action is given by a section 
\[
\rho^G: Y\times_{BG}\mathfrak g_{BG}\to T_pY\times_{BG}\mathfrak g_{BG}
\]
and its restriction $\rho$ on $\mathfrak g\times X$.
We give the construction of $\rho^G$.

Since \(p:Y\to BG\) is a Kan fibration, the map
\[
        (s_Y,p_1):Y_1\longrightarrow Y_0\times G
\]
is a surjective submersion.  Since \(X=p_1^{-1}(e)\) is \'{e}tale, one has
\[
        \dim Y_1=\dim Y_0+\dim G .
\]
Hence
\[
        (s_Y,p_1):Y_1\to Y_0\times G
\]
is a local diffeomorphism.  Consequently, for every
\(\gamma\in Y_1\), the differential induces a linear isomorphism
\begin{equation}\label{E_ds_etale}
        ds_Y:
        \ker(dp_1)_\gamma
        \overset{\cong}{\longrightarrow}
        T_{s_Y(\gamma)}Y_0 .
\end{equation}
Similarly, 
\begin{equation}\label{E_dt_etale}
        dt_Y:
        \ker(dp_1)_\gamma
        \overset{\cong}{\longrightarrow}
        T_{t_Y(\gamma)}Y_0 
\end{equation}
is also a linear isomorphism.
This is the basic fact used below.

\medskip

First we define the object part of the infinitesimal action.  Locally, choose
the inverse branch of
\[
        (s_Y,p_1):Y_1\to Y_0\times G
\]
through the identity arrows.  Thus, near \((y,e)\), we have a smooth map
\[
        \sigma:U\times V\longrightarrow Y_1
\]
such that
\[
        s_Y(\sigma(y,g))=y,
        \qquad
        p_1(\sigma(y,g))=g,
        \qquad
        \sigma(y,e)=1_y .
\]
For \(\xi\in\mathfrak g\), define
\[
        \rho_0(y,\xi)
        :=
        \left.\frac{d}{dt}\right|_{t=0}
        t_Y\bigl(\sigma(y,\exp(t\xi))\bigr)
        \in T_yY_0=T_yX_0 .
\]
Since the inverse branch is uniquely determined near the identity arrow, these
local formulas define a well-defined smooth map
\[
        \rho_0:X_0\times\mathfrak g\longrightarrow TX_0,
        \qquad
        (x,\xi)\longmapsto \rho_0(x,\xi).
\]
It is linear in the \(\mathfrak g\)-variable.

Because \(X\) is \'{e}tale, the source map
\[
        s_X:X_1\to X_0
\]
is a local diffeomorphism.  Hence, for every arrow
\[
        \alpha:x\to y
\]
in \(X\), there is a unique vector
\[
        \rho_1(\alpha,\xi)\in T_\alpha X_1
\]
satisfying
\[
        ds_X\bigl(\rho_1(\alpha,\xi)\bigr)
        =
        \rho_0(x,\xi).
\]
The compatibility condition for the infinitesimal action on the orbifold is
\[
        dt_X\bigl(\rho_1(\alpha,\xi)\bigr)
        =
        \rho_0(y,\xi).
\]
This follows from Proposition \ref{P_inf_action}.
Equivalently, \(\rho_0(\xi)\) is an \(X\)-invariant vector field.  Thus
\(\rho\) may be regarded as a section
of the vector bundle
\[
       TX\times \mathfrak g\longrightarrow X\times \mathfrak g,
\]
with
\[
        (x,\xi)
        \mapsto
        \bigl(x,\xi,\rho_0(x,\xi)\bigr),\;\;\;
(\alpha,\xi)
        \mapsto
        \bigl(\alpha,\xi,\rho_1(\alpha,\xi)\bigr).
\]

The infinitesimal action condition is that
\[
        \xi\longmapsto \hat\rho(\xi),
\]
where 
\[
\hat\rho_0(x)=\rho_0(x,\xi),\;\;\;
\hat\rho_1(\alpha)=\rho_1(\alpha,\xi).
\]
This is a Lie algebra homomorphism:
\[
        [\hat\rho(\xi),\hat\rho(\eta)]
        =
        \hat\rho([\xi,\eta]).
\]

\medskip

Now we describe the \(G\)-equivariant extension of this section at the
Kan-fibration level.

An arrow \def \Ad{\mathrm{Ad}}
\[
        (\gamma,\xi)\in (Y\times_{BG}\mathfrak g_{BG})_1
\]
with
\[
        \gamma:x\to y,
        \qquad
        p_1(\gamma)=g,
\]
goes from
\[
        (x,\xi)
\;\;\;\mbox{ 
to }\;\;\;
        (y,\Ad_g\xi).
\]
We now define $\rho^G$.
On objects, \(
        \rho^G_0=\rho_0.
\)
On arrows,  define
\[
        \rho^G_1(\gamma,\xi)
        \in
        \ker(dp_1)_\gamma
\]
to be the unique vector satisfying
\[
        ds_Y\bigl(\rho^G_1(\gamma,\xi)\bigr)
        =
        \rho_0(x,\xi).
\]
Equivalently,
\[
        \rho^G_1(\gamma,\xi)
        =
        \left(ds_Y|_{\ker(dp_1)_\gamma}\right)^{-1}
        \bigl(\rho_0(s_Y(\gamma),\xi)\bigr).
\]
This is well-defined because
\[
        ds_Y:\ker(dp_1)_\gamma
        \overset{\cong}{\longrightarrow}
        T_xY_0
\]
is a linear isomorphism.

The \(G\)-equivariance of the infinitesimal action is the identity
\begin{equation}\label{E_inf_action}
        dt_Y\bigl(\rho^G_1(\gamma,\xi)\bigr)
        =
        \rho_0(y,\Ad_g\xi)
\end{equation}
which is proved in Proposition \ref{P_inf_action}.
Thus the arrow
\[
        \rho^G_1(\gamma,\xi)\in T_pY_1
\]
connects the vector
\[
        \rho_0(x,\xi)\in T_xY_0
\]
to the vector
\[
        \rho_0(y,\Ad_g\xi)\in T_yY_0 .
\]

\begin{prop}[Equivariance of the infinitesimal action]
\label{P_inf_action} The property given by \eqref{E_inf_action} holds.
\end{prop}

\begin{proof}
Let
\(
        h_t:=\exp(t\xi).
\)
and
\(
        g h_t g^{-1}
        =
        \exp(t\Ad_g\xi).
\)
By the definition of \(\rho_0\) at \(y\), there is a unique local path of arrows
\[
        b_t:y\to y_t
\]
such that
\[
        b_0=1_y,\qquad
        s_Y(b_t)=y,\qquad
        p_1(b_t)=g h_t g^{-1}.
\]
Thus
\[
        \rho_0(y,\Ad_g\xi)
        =
        \left.\frac{d}{dt}\right|_{t=0}t_Y(b_t)
        =
        \left.\frac{d}{dt}\right|_{t=0}y_t .
\]

Consider the path of arrows
\[
        \gamma_t:=b_t\circ \gamma\circ a_t^{-1}.
\]
This is an arrow
\[
        \gamma_t:x_t\to y_t .
\]
Its \(G\)-label is constant:
\[
\begin{aligned}
        p_1(\gamma_t)
        &=
        p_1(b_t)\,p_1(\gamma)\,p_1(a_t)^{-1}        \\
        &=
        (g h_t g^{-1})\,g\,h_t^{-1}                 \\
        &=
        g .
\end{aligned}
\]
Hence
\[
        \dot\gamma_0
        :=
        \left.\frac{d}{dt}\right|_{t=0}\gamma_t
        \in
        \ker(dp_1)_\gamma .
\]

Moreover,
\[
        ds_Y(\dot\gamma_0)
        =
        \left.\frac{d}{dt}\right|_{t=0}s_Y(\gamma_t)
        =
        \left.\frac{d}{dt}\right|_{t=0}x_t
        =
        \rho_0(x,\xi).
\]
By the uniqueness of the vector in \(\ker(dp_1)_\gamma\) with prescribed
source derivative, we have
\[
        \rho^G_1(\gamma,\xi)=\dot\gamma_0.
\]
Taking target derivatives gives
\[
\begin{aligned}
        dt_Y\bigl(\rho^G_1(\gamma,\xi)\bigr)
        &=
        dt_Y(\dot\gamma_0)                                      \\
        &=
        \left.\frac{d}{dt}\right|_{t=0}t_Y(\gamma_t)              \\
        &=
        \left.\frac{d}{dt}\right|_{t=0}y_t                        \\
        &=
        \rho_0(y,\Ad_g\xi).
\end{aligned}
\]
This proves the desired equivariance identity.
\end{proof}

We already define a section 
$\rho: \mathfrak g\times X\to \mathfrak g\times TX$, or equivalently a bundle map $\mathfrak g\times X\to TX$ still denoted by $\rho$. For each $\xi\in \mathfrak g$, define the vector field 
\begin{equation}\label{E_inf_vectorfield}
\xi^\sharp: X\to \{\xi\}\times X\xrightarrow{\rho} TX.
\end{equation}

\subsection{Equivariant cohomolgy on orbifolds}

Let
\[
X\hookrightarrow Y\xrightarrow{p}BG
\]
be a Kan fibration where X is an orbifold groupoid, namely a proper \'{e}tale Lie groupoid. The main result is the following.

\begin{theorem}
$\Omega^\bullet_\bas(X)$ is a $G^\ast$ algebra.
\end{theorem}

\begin{proof} We first explain the $G$ action on $\Omega^\bullet_\bas(X)$.
For \(a\in G\), define
\[
P_a:=p_1^{-1}(a)\subseteq Y_1.
\]
Then \(P_a\) is an \(X\)-\(X\) bibundle, with moment maps
\[
s_Y,t_Y:P_a\to X_{0}.
\]
Hence \(P_a\) induces an automorphism
\[
\rho(a):\Omega^\bullet_{\bas}(X)\to \Omega^\bullet_{\bas}(X).
\]
The multiplication in \(Y_1\) gives canonical bibundle isomorphisms
\[
P_a\otimes_{X_0} P_b\cong P_{ab},
\]
and therefore
\[
\rho(a)\rho(b)=\rho(ab).
\]
For any $\xi\in \mathfrak g$,
define operators on \(\Omega^\bullet_{\bas}(X)\) by
\[
\iota_\xi:=\iota_{\xi^\#},
\qquad
\mathcal L_\xi:=\mathcal L_{\xi^\#}.
\]
The usual Cartan identities, which are available on orbifolds, give
\[
[\dd,\iota_\xi]=\mathcal L_\xi,
\qquad
[\dd,\mathcal L_\xi]=0,
\qquad
[\iota_\xi,\iota_\eta]=0.
\]
Since
\[
\mathfrak g\to\Vect(X),
\qquad
\xi\mapsto \xi^\#
\]
is a Lie algebra morphism, one also has
\[
[\mathcal L_\xi,\iota_\eta]=\iota_{[\xi,\eta]},
\qquad
[\mathcal L_\xi,\mathcal L_\eta]=\mathcal L_{[\xi,\eta]}.
\] 
Moreover, the \(G\)-equivariance of the construction gives
\[
\rho(a)\mathcal L_\xi\rho(a^{-1})
=
\mathcal L_{\Ad_a\xi},
\qquad
\rho(a)\iota_\xi\rho(a^{-1})
=
\iota_{\Ad_a\xi},
\qquad
\rho(a)d\rho(a^{-1})=d.
\]
Therefore
\[
\Omega^\bullet_{\bas}(X)
\]
is a \(G^\ast\)-algebra.
\end{proof}

\begin{definition}
    Let $X$ be a proper \'etale Lie groupoid, i.e, an orbifold. Suppose it admits a weak $G$ action. 
    The \emph{Weil model} of the \(G\)-equivariant cohomology of the orbifold \(X\) is
\[
H_G^\bullet(X)
:=
H^\bullet\left(
\left(\Omega^\bullet_{\bas}(X)\otimes W(\mathfrak g)\right)_{\bas},
d\otimes 1+1\otimes d_W
\right).
\]

The \emph{Cartan model} is
\[
C_G^\bullet(X)
:=
\left(S(\mathfrak g^\vee)\otimes \Omega^\bullet_{\bas}(X)\right)^G.
\]
Its differential is
\[
d_G
=
1\otimes d-\sum_a u^a\otimes \iota_{\xi_a},
\]
where \(\{\xi_a\}\) is a basis of \(\mathfrak g\), and \(\{u^a\}\) is the
corresponding degree \(2\) basis of \(S^1(\mathfrak g^\vee)\). Thus
\[
H_G^\bullet(X)
=
H^\bullet\left(C_G^\bullet(X),d_G\right).
\]

\end{definition}

The Weil and Cartan models are naturally isomorphic:
\[
H^\bullet\left(
\left(\Omega^\bullet_{\bas}(X)\otimes W(\mathfrak g)\right)_{\bas}
\right)
\cong
H^\bullet\left(
\left(S(\mathfrak g^\vee)\otimes \Omega^\bullet_{\bas}(X)\right)^G,
d_G
\right).
\]

\begin{remark}
There are two  ways to put a \(G^\ast\)-algebra structure on
\[
\Omega^\bullet_{\bas}(X)
\]
which can be shown to agree.

The first method replaces \(X\) by a Morita equivalent Lie groupoid
\(
Z
\)
which carries a strict free \(G\)-action. Then
\(
\Omega^\bullet_{\bas}(Z)
\)
has the usual \(G^\ast\)-algebra structure coming from this strict action, and
this structure is transported to
\(
\Omega^\bullet_{\bas}(X)
\)
through the Morita equivalence
\(
X\simeq Z.
\)

The second method is intrinsic to the orbifold. Since \(X\) is \'{e}tale, the Kan
fibration gives an infinitesimal action
\[
\mathfrak g\longrightarrow \operatorname{Vect}(X),
\qquad
\xi\longmapsto \xi^\#.
\]
One then  defines
\[
\iota_\xi:=\iota_{\xi^\#},
\qquad
\mathcal L_\xi:=\mathcal L_{\xi^\#}
\]
directly on
\(
\Omega^\bullet_{\bas}(X)
\) via {\bf the differential geometry} on orbifolds.
\end{remark}

\subsection{The fixed subgroupoid of a weak \(G\)-action}

Let
\[
        p:Y\longrightarrow BG
\]
be a Kan fibration of Lie groupoids with fiber $X$. 
For \(x\in X_0=Y_0\), recall that we have an exact sequence \eqref{E_isotropy} of $\Gamma_x$ and $\widehat{\Gamma}_x$.

\begin{defn}
A point \(x\in X_0\) is called \(G\)-fixed if
\(
        p_1:\widehat\Gamma_x\to G
\)
is surjective. We denote the set of \(G\)-fixed objects by
\(
        X_0^G \).

\end{defn}

Equivalently, \(X_0^G\) is the set of objects whose orbit \([x]\in |X|\) is
fixed by the induced \(G\)-action on the orbit space \(|X|\). Hence, we have the following easy consequence.
\begin{lemma}[Saturation of the fixed-object set]
The subset \(X_0^G\subset X_0=Y_0\) is saturated under both \(X\) and \(Y\).
\end{lemma}
Set 
\begin{equation}
    X^G=X^G_0\rtimes X
\end{equation}
be the $G$-fixed subgroupoid, we have the Kan fibration 
\[
X^G\hookrightarrow 
X^G_0\rtimes Y\to BG. 
\]

\begin{theorem}[Fixed subgroupoid for a weak compact connected action]
Assume that \(X\) is a proper \'{e}tale Lie groupoid and that
\(G\) is compact and connected. Then 
\(X_0^G\subset X_0\) is an embedded submanifold. Hence
$X^G$ is a full, Lie subgroupoid of $X$.
\end{theorem}

\begin{proof}
Let
\(
        x\in X_0^G .
\)
Since \(Y\) is proper, the slice theorem for proper Lie groupoids gives a
sufficiently small neighborhood \(U_x\subset X_0=Y_0\) of \(x\), invariant under
\(\widehat\Gamma_x\), such that the restriction of \(Y\) over \(U_x\) is
modeled by the action groupoid
\(
        \widehat\Gamma_x\ltimes U_x .
\)
Similarly, the restriction of \(X\) over \(U_x\) is modeled by
\(
        \Gamma_x\ltimes U_x ,
\).
Since \(X\) is an orbifold groupoid, \(\Gamma_x\) is finite.

Let
\(
        \widehat\Gamma_x^0
\)
be the identity component of \(\widehat\Gamma_x\). Since $G$ is connected, 
\(
        p_1(\widehat\Gamma_x^0)=G,
\)
and
is a finite covering of compact connected Lie groups.

We claim that
\begin{equation}\label{E_1}
        X_0^G\cap U_x
        =
        U_x^{\widehat\Gamma_x^0}.
\end{equation}
First suppose
\[
        y\in U_x^{\widehat\Gamma_x^0}.
\]
Then
\(
        \widehat\Gamma_x^0\subset \widehat\Gamma_y,
\)
Thus
\(
        y\in X_0^G.
\)
This proves "$\supset$" of 
\eqref{E_1}.

Conversely, suppose
\[
        y\in X_0^G\cap U_x .
\]
Then the composition
\[
        \widehat\Gamma_y
        \hookrightarrow
        \widehat\Gamma_x
        \xrightarrow{p_x}
        G
\]
is surjective.   Since \(G\) is connected, this implies $\widehat\Gamma_y^0\subset \widehat\Gamma_x^0$ and
\[
        p_1(\widehat\Gamma_y^0)=p_1(\widehat\Gamma_x^0)=G.
\]
We have 
\[\dim(\widehat\Gamma_y^0)=
\dim (\widehat\Gamma_x^0)=\dim G.
\]
By the connectedness assumption, 
\[
        \widehat\Gamma_y^0
        =
        \widehat\Gamma_x^0.
\]
In particular,
\[
        \widehat\Gamma_x^0\subset \widehat\Gamma_y.
\]
Therefore \(y\) is fixed by \(\widehat\Gamma_x^0\), so
we prove "$\subset$" of \eqref{E_1}. This verifies 
\eqref{E_1}.

Since \(\widehat\Gamma_x^0\) is a compact Lie group acting smoothly on the slice
\(U_x\), its fixed-point set
\[
        U_x^{\widehat\Gamma_x^0}
\]
is an embedded submanifold of \(U_x\).  Therefore \(X_0^G\) is locally an
embedded submanifold near every point \(x\in X_0^G\).  Hence
       \( X_0^G\)
is an embedded submanifold of $X_0$.
\end{proof}

The connectedness assumption is crucial. The following example shows that $X^G$ is not smooth if $G$ is a finite group.

\begin{example}[Failure for finite discrete \(G\)]
Let
\[
H=\mathbb Z_2
\]
act on \(\mathbb R^2\) by
\[
h\cdot(x,y)=(-x,-y),
\]
and let
\[
X=H\ltimes \mathbb R^2
\]
be the corresponding orbifold groupoid. Let
\[
G=\mathbb Z_2=\{e,\sigma\}
\]
act on \(\mathbb R^2\) by
\[
\sigma(x,y)=(x,-y).
\]
This action commutes with the \(H\)-action, hence gives a strict \(G\)-action on \(X\).

For a point \((x,y)\in \mathbb R^2\), the nontrivial element \(\sigma\in G\) fixes its
\(X\)-orbit if and only if there exists \(h\in H\) such that
\[
h\cdot \sigma(x,y)=(x,y).
\]
If \(h=e\), then
\[
(x,-y)=(x,y),
\]
so \(y=0\). If \(h\neq e\), then
\[
(-x,y)=(x,y),
\]
so \(x=0\). Hence
\[
X^G_0
=
\{(x,y)\in\mathbb R^2\mid x=0\}
\cup
\{(x,y)\in\mathbb R^2\mid y=0\}.
\]
This is the union of the two coordinate axes, which is not a smooth submanifold at the
origin. Therefore \(X^G\) is not a Lie groupoid.
\end{example}

 \def \SN{\mathsf N}

Let $\tau$ be a $G$-invariant metric on $X$ (cf. Proposition \ref{P_inv_metric}).  
  Let $N$ be the normal bundle of $X^G$ in $X$ with respect to the metric $\tau$. It is also $G$-equivariant. 
Be precisely, let
\[
\pi_{N,0}: N_0\to X^{G}_0
\]
be a vector bundle.  It admits both $X$ and $Y$ action, so
$$
N= N_0\rtimes X,\;\;\; 
N\rtimes B G=N_0\rtimes Y.
$$
The later is a bundle over $X^G\rtimes B G=X^{G}_0\rtimes Y$. 
\begin{lemma}
The normal bundle $N$ of $X^G$ in $X$ is $G$ equivariant bundle given by 
the Kan fibration
\[
N\hookrightarrow N_0\rtimes Y\to BG.
\]
\end{lemma}
\subsection{Integral localization formulae}

We first show the existence of $G$-equivariant Thom forms for $G$ equivariant vector bundle over orbifolds.

\begin{prop}[Existence of a \(G\)-equivariant Thom form]
 Let $E$ be a $G$ equivariant oriented vector bundle of rank $r$ over orbifold $X$ given by the diagram \eqref{E_equi_VB}. 
Assume that  \(G\) is compact and the bundle
 \(E\to X\) is oriented of real rank \(r\), and that the orientation
is preserved by the \(G\)-equivariant structure.

Then there exists an equivariant Thom form
\[
        \tau_G(E)\in
        \Omega^r_{G,\bas,\mathrm{cv}}(E)
        :=
        \left(
        S(\mathfrak g^\vee)\otimes
        \Omega^\bullet_{\bas,\mathrm{cv}}(E)
        \right)^G ,
\]
where \(\mathrm{cv}\) means compact vertical support, such that
\[
        d_G\tau_G(E)=0 \;\;\;
\mbox{ 
and }\;\;\;
        (\pi_E)_*\tau_G(E)=1
\]
in
\(
        \Omega^\bullet_{G,\bas}(X).
\)
Moreover, if
\(
        0:X\longrightarrow E
\)
is the zero section, then
\(
        0^*\tau_G(E)=e_G(E),
\)
the equivariant Euler form of \(E\).
\end{prop}

\begin{proof}
We construct \(\Theta_E\) by the frame-bundle construction, following the
Mathai--Quillen and Guillemin--Sternberg construction of equivariant Thom forms~\cite{MathaiQuillen1986,GuilleminSternberg1999}.

Let
\[
        E_0\longrightarrow X_0
\]
be the vector bundle on the object space representing \(E\to X\).  Choose a
\(G\)-invariant fiber metric on \(E\).  This is possible because \(G\) is
compact.

Let
\[
        V=\mathbb R^r,
        \qquad
        K=SO(V).
\]
Let
\[
        P^0:=\operatorname{Fr}^+(E_0)
\]
be the oriented orthonormal frame bundle of \(E_0\).  Then \(P^0\to X_0\) is a
principal \(K\)-bundle and
\[
        E_0\cong P^0\times_K V.
\]

The vector bundle \(\mathcal E\to Y\) is read as an action of the total groupoid
\(Y\) on \(E_0\).  Since the metric and orientation are \(Y\)-invariant, this
action lifts to \(P^0\).  The right \(K\)-action on \(P^0\) commutes with the
\(Y\)-action.  Equivalently, if we write schematically
\[
        Y=X\rtimes BG
\]
for the groupoid representing the weak \(G\)-action, then the \(K\)-action
commutes with the \(X\rtimes BG\)-action on \(P^0\).

Let
\[
        \mathsf P:=P^0\rtimes Z
\]
be the corresponding principal \(K\)-bundle groupoid over \(X\).  Then
\[
        K\backslash \mathsf P\cong X,
\]
and the associated vector bundle is
\[
        E\cong \mathsf P\times_K V.
\]
At the equivariant, or Kan-fibration, level this is expressed by the quotient
presentation
\[
        BK\ltimes(\mathsf P\times V)\rtimes BG
        \longrightarrow
        X\rtimes BG,
\]
where \(K\) acts on \(V\) by the standard representation and \(G\) acts through
the given weak action on \(\mathsf P\).

There is an evident projection morphism
\[
        f:
        BK\ltimes(\mathsf P\times V)\rtimes BG
        \longrightarrow
        V\rtimes B(K\times G),
\]
given on objects by
\[
        (p,v)\longmapsto v.
\]
Here \(K\) acts on \(V\) by the standard representation, and \(G\) acts
trivially on \(V\).

This induces a pullback map on equivariant forms
\[
        f^*:
        \Omega_{K\times G}^\bullet(V)
        \longrightarrow
        \Omega_{K\times G}^\bullet(\mathsf P\times V).
\]

Next choose a \(K\)-connection on the principal \(K\)-bundle
\[
        \mathsf P\longrightarrow Z .
\]
Since the \(G\)-action commutes with the \(K\)-action and \(G\) is compact, we
may choose this connection to be \(G\)-invariant.  The Cartan--Weil homomorphism
associated to this connection gives a chain map
\[
        \Pi:
        \Omega_{K\times G}^\bullet(\mathsf P\times V)
        \longrightarrow
        \Omega_G^\bullet\bigl(\Omega^\bullet(\mathsf P\times V)_{K,\bas}\bigr).
\]
This is the usual Guillemin--Sternberg map~\cite{GuilleminSternberg1999}: it converts \(K\)-equivariant forms
on a free \(K\)-space into \(K\)-basic forms, while retaining the remaining
\(G\)-equivariant structure.

Since the \(K\)-action on \(\mathsf P\times V\) is free, the basic forms are
identified with forms on the quotient:
\[
        \Omega^\bullet(\mathsf P\times V)_{K,\bas}
        \cong
        \Omega^\bullet(\mathsf P\times_K V).
\]
Thus
\[
        \Omega_G^\bullet\bigl(\Omega^\bullet(\mathsf P\times V)_{K,\bas}\bigr)
        \cong
        \Omega_G^\bullet(\mathsf P\times_K V)
        =
        \Omega_G^\bullet(E).
\]

Therefore we have a natural chain map
\[
        \Omega_{K\times G}^\bullet(V)
        \xrightarrow{f^*}
        \Omega_{K\times G}^\bullet(\mathsf P\times V)
        \xrightarrow{\Pi}
        \Omega_G^\bullet\bigl(\Omega^\bullet(\mathsf P\times V)_{K,\bas}\bigr)
        \xrightarrow{\cong}
        \Omega_G^\bullet(E).
\]
Denote this composite by
\[
        C_{\mathsf P}:
        \Omega_{K\times G}^\bullet(V)
        \longrightarrow
        \Omega_G^\bullet(E).
\]

Let
\[
        \Theta_{K\times G}\in \Omega^r_{K\times G,\mathrm{cv}}(V)
\]
be the universal equivariant Thom form of the oriented \(K\)-representation
\(V\).  Since \(G\) acts trivially on \(V\), this is just the usual
\(K\)-equivariant Mathai--Quillen Thom form, regarded as a
\((K\times G)\)-equivariant form.  We use the compactly supported modification
of the Mathai--Quillen form as in Guillemin--Sternberg.  Thus
\[
        d_{K\times G}\Theta_{K\times G}=0,
\]
and
\[
        \int_V \Theta_{K\times G}=1.
\]

Define
\[
        \Theta_E
        :=
        C_{\mathsf P}\bigl(\Theta_{K\times G}\bigr)
        \in
        \Omega^r_{G,\bas,\mathrm{cv}}(E).
\]
Because \(C_{\mathsf P}\) is a chain map, we have
\[
        d_G\Theta_E
        =
        d_G C_{\mathsf P}(\Theta_{K\times G})
        =
        C_{\mathsf P}(d_{K\times G}\Theta_{K\times G})
        =
        0.
\]
The compact vertical support of \(\Theta_E\) follows from the compact vertical
support of \(\Theta_{K\times G}\) on \(V\).  Locally on \(X\), the bundle
\(E\to X\) is identified with an associated bundle
\[
        P^0\times_K V,
\]
and \(\Theta_E\) is obtained from the universal Thom form on \(V\).  Therefore
fiber integration gives
\[
        (\pi_E)_*\Theta_E=1.
\]

Finally, pulling back along the zero section
\[
        0:X\to E
\]
gives the equivariant Euler form:
\[
        0^*\Theta_E=e_G(E).
\]
This is the standard Mathai--Quillen identity, transported through the
Cartan--Weil homomorphism and the quotient identification
\[
        \mathsf P\times_K V\cong E.
\]
Hence \(\Theta_E\) is a \(G\)-equivariant Thom form of \(E\to X\).
\end{proof}

\def \orb{\mathrm{orb}}

\begin{theorem}[Local localization formula on an equivariant vector bundle]
Let \(G=S^1\), and let
\[
        \pi:E\longrightarrow X
\]
be an oriented \(G\)-equivariant orbifold vector bundle over a connected, compact oriented
orbifold \(X\). Let
\[
        i:X\longrightarrow E
\]
be the zero section.

Assume that the fixed-point locus of the total space \(E\) is precisely the
zero section, i.e, 
\(
        E^G=i(X).
\)

Let
\[
        \mu\in \Omega^\bullet_{G,\bas,\mathrm{cv}}(E)
\]
be an equivariantly closed form with compact vertical support.
Then, after inverting the equivariant parameter \(u\), one has 
\[
        \int_E^{\orb}\mu
        =
        \int_X^{\orb}
        \frac{i^*\mu}{e_G(E)}      
\]
as an identity in
\(
        \mathbb R[u,u^{-1}].
\)
\end{theorem}
\begin{proof}
Let
\(
        r=\operatorname{rank}_{\mathbb R}E .
\)
Since \(E\to X\) is oriented and \(G\)-equivariant, there exists an equivariant
Thom form
\(
        \tau_G(E)\).

The equivariant Thom isomorphism is given by
\[
        i_*:
        H^\bullet_{G,\bas}(X)
        \longrightarrow
        H^{\bullet+r}_{G,\bas,\mathrm{cv}}(E),
        \qquad
        \beta\longmapsto
        \pi^*\beta\wedge \tau_G(E).
\]
It satisfies the self-intersection formula
\[
        i^*i_*(\beta)=\beta\, e_G(E).
\]

Let \([\mu]\in H^\bullet_{G,\bas,\mathrm{cv}}(E)\) be the cohomology class of
\(\mu\). By the equivariant Thom isomorphism, there exists a unique class
\[
        \beta\in H^{\bullet-r}_{G,\bas}(X)
\]
such that
\(
        [\mu]=i_*(\beta).
\)
Pulling back by the zero section gives
\[
        i^*[\mu]
        =
        i^*i_*(\beta)
        =
        \beta\, e_G(E).
\]
 In
localized equivariant cohomology,
\[
        \beta
        =
        \frac{i^*[\mu]}{e_G(E)}.
\]
Thus
\[
        [\mu]
        =
        i_*\left(\frac{i^*[\mu]}{e_G(E)}\right).
\]

Now integrate both sides over \(E\). Since integration is compatible with the
Gysin map, we have
\[
        \int_E^{\orb} i_*(\beta)
        =
        \int_X^{\orb}\beta .
\]
Therefore
\[
\begin{aligned}
        \int_E^{\orb}\mu
        &=
        \int_E^{\orb}
        i_*\left(\frac{i^*\mu}{e_G(E)}\right)        \\
        &=
        \int_X^{\orb}
        \frac{i^*\mu}{e_G(E)} .
\end{aligned}
\]
This proves the formula.
\end{proof}

\begin{theorem}[Integral localization formula for compact orbifolds]\label{T_local_formula}
Let
\[
        p:Y\longrightarrow BS^1
\]
be a Kan fibration of Lie groupoids with compact oriented orbifold fiber $X$.
Assume that the weak \(S^1\)-action preserves the orientation of \(X\). Let
\[
        X^{S^1}=\bigsqcup_F F
\]
be the fixed-point orbifold, written as a disjoint union of connected fixed
components. For each component, let
\[
        i_F:F\hookrightarrow X
\]
be the inclusion, and let
\[
        N_F:=TX|_F/TF
\]
be the normal orbibundle.

Let
\[
        \Omega^\bullet_{S^1,\bas}(X)
        =
        \bigl(\Omega^\bullet_{\bas}(X)[u]\bigr)^{S^1},
        \qquad
        \deg u=2,
\]
be the Cartan model of the equivariant basic complex, with differential
\[
        d_{S^1}=d-u\,\iota_\xi,
\]
where \(\xi\) is the generator of
\(
    \operatorname{Lie}(S^1).
\)

Then for every equivariantly closed form
\[
        \alpha\in \Omega^\bullet_{S^1,\bas}(X),
        \qquad
        d_{S^1}\alpha=0,
\]
one has the localization formula
\[
        \int_X^{\orb}\alpha
        =
        \sum_{F\subset X^{S^1}}
        \int_F^{\orb}
        \frac{i_F^*\alpha}{e_{S^1}(N_F)}
\]
as an identity in the localized ring
\(
        \mathbb R[u,u^{-1}].
\)
\end{theorem}

 \begin{proof}
We prove the formula first for $\alpha$ being a homogeneous equivariantly closed form
\[
        \mu\in \Omega^\bullet_{S^1,\bas}(X),
        \qquad
        d_{S^1}\mu=0.
\]
The general case follows by decomposing into homogeneous components. If the
total degree of \(\mu\) is smaller than \(\dim X\), we multiply by a sufficiently
large power of \(u\), prove the formula for \(u^N\mu\), and then divide by
\(u^N\) in the localized ring \(\mathbb R[u,u^{-1}]\). Hence we may assume that
the total degree of \(\mu\) is at least
\[
        d:=\dim X .
\]

By the \(S^1\)-equivariant tubular neighborhood theorem for orbifolds, there is a
\(S^1\)-invariant tubular neighborhood
\(
        U
\)
of \(X^{S^1}\) in \(X\), equivariantly identified with a neighborhood of the zero
section in the normal orbibundle \(N\).  We shall use the standard local
localization formula, obtained from the equivariant Thom form of \(N\), or
equivalently from formula (10.20) in Guillemin--Sternberg~\cite[Formula~(10.20)]{GuilleminSternberg1999}:
if
\[
        \eta\in \Omega^\bullet_{S^1,\bas,\mathrm{c}}(U)
\]
is \(d_{S^1}\)-closed and supported in \(U\), then
\[
        \int_{U}^{\orb}\eta
        =
        \int_{X^{S^1}}^{\orb}
        \frac{i^*\eta}{e_{S^1}(N)} .
\]
It is invertible
after localizing at \(u\), since \(N\) has no zero-weight summand along the
fixed locus.

Now set
\[
        X^c:=X\setminus X^{S^1} .
\]
On \(X^c\), the \(S^1\)-action is locally free.  Indeed, at a point outside the
fixed locus, the stabilizer is a proper closed subgroup of \(S^1\), hence is
finite.  Therefore the locally free comparison theorem gives
\[
        H^\bullet_{S^1,\bas}(X^c)
        \cong
        H^\bullet_{\bas}(X^c\rtimes BS^1).
\]
The quotient orbifold
\(
        X^c\rtimes BS^1
\)
has dimension
\(
        d-1.
\)
Hence
\[
        H^m_{S^1,\bas}(X^c)=0
        \qquad
        \text{for all }m\geq d .
\]
Since the total degree of \(\mu\) is at least \(d\), the restriction of \(\mu\)
to \(X^c\) is \(d_{S^1}\)-exact.  Thus there exists
\(
        \nu
\)
such that
\(
        \mu|_{X^c}=d_{S^1}\nu .
\)

Choose a \(S^1\)-invariant smooth function
\(
        \rho\in C^\infty(X)
\)
such that
\(
        \operatorname{supp}(\rho)\subset U,
\)
and such that
\(
        \rho=1
\)
on a smaller \(S^1\)-invariant tubular neighborhood
\[
         U'\subset  U
\]
of \(X^{S^1}\).  

On \(X^c\), define
\[
        \nu':=(1-\rho)\nu .
\]
Because \(1-\rho=0\) near \(X^{S^1}\), the form \(\nu'\) extends smoothly by zero
across \(X^{S^1}\) to a global equivariant form on \(X\).  Define
\[
        \mu':=\mu-d_{S^1}\nu' .
\]
Then
\[
        d_{S^1}\mu'=0,
\]
and
\[
        [\mu']=[\mu]
\]
in equivariant cohomology.  Moreover, \(\mu'\) is supported in \( U\).
On the smaller neighborhood \(U'\), we have 
\[
        \nu'=0,
\]
and
\[
        \mu'=\mu.
\]
In particular,
\[
        i^*\mu'=i^*\mu .
\]

Since \(X\) is compact and has no boundary, orbifold Stokes' theorem gives \def \orb{\mathrm{orb}}
\[
        \int_X^{\orb} d_{S^1}\nu'=0.
\]
Hence
\[
        \int_X^{\orb}\mu
        =
        \int_X^{\orb}\mu' .
\]
Because \(\mu'\) is supported in \(U\),
\[
        \int_X^{\orb}\mu'
        =
        \int_{ U}^{\orb}\mu' .
\]
Applying the local localization formula to \(\mu'\), we obtain
\[
        \int_{U}^{\orb}\mu'
        =
        \int_{X^{S^1}}^{\orb}
        \frac{i^*\mu'}{e_{S^1}(N)} .
\]
Since \(i^*\mu'=i^*\mu\), this becomes
\[
        \int_X^{\orb}\mu
        =
        \int_{X^{S^1}}^{\orb}
        \frac{i^*\mu}{e_{S^1}(N)} .
\]

Finally, writing
\[
        X^{S^1}=\bigsqcup_F F
\]
as the disjoint union of connected fixed components, and writing
\[
        i_F:F\hookrightarrow X
\]
for the inclusions, the last identity is exactly
\[
        \int_X^{\orb}\mu
        =
        \sum_{F\subset X^{S^1}}
        \int_F^{\orb}
        \frac{i_F^*\mu}{e_{S^1}(N_F)} ,
\]
where
\[
        N_F:=N_{F/X}.
\]
This proves the localization formula in
\(
        \mathbb R[u,u^{-1}] .
\)
\end{proof}

\section{Equivariant Chen--Ruan cohomology ring}\label{Sec_equiv_CR} 

Chen--Ruan cohomology is the standard ``stringy'' cohomology ring of an almost complex orbifold. It was introduced by Chen and Ruan~\cite{ChenRuan2004} as a new cohomology theory for orbifolds, motivated by orbifold string theory and degree-zero genus-zero orbifold Gromov--Witten theory. This Gromov-Witten theory is based on the moduli groupoids of stable holomorphic curves in a symplectic orbifold and is a quantum field theory on $H^\ast_\bas(IX)$ with proper degree shifting, where $IX$ is the inertia orbifold of $X$. Hence, it is a differential geometry type theory.  With the $G$ action, we have already equivariantize the cohomology $H^\ast_{\bas}(X)$ in the previous section. In this section, we show that $IX$ admits an induced $G$ action, hence we have $H^\ast_{G,\bas}(IX)$. The core of the Chen-Ruan theory is the 3-point function on $H^\ast_{\bas}(IX)$. We give its equivariant version in \S\ref{Subsec_equiv_3pt}, hence yield the equivariant Chen-Ruan cohomology ring.

\subsection{Review of the Chen--Ruan cohomology ring}

\def \IMor{\mathrm{IMor}}

We briefly review the construction of the Chen--Ruan cohomology ring of an
almost complex orbifold.  Let $X$
be a proper \'{e}tale Lie groupoid presenting an orbifold.  For a point
\(x\in X^0\), we denote its isotropy group by $\Gamma_x$.

For a tuple
\[
        \vec k=(k_1,\ldots,k_m),
\]
let \(\mathbb Z_{\vec k}\) be the group generated by
\(\lambda_1,\ldots,\lambda_m\) with relations
\[
        \lambda_1\cdots\lambda_m=1,
        \qquad
        \lambda_i^{k_i}=1
        \quad
        \text{for }i=1,\ldots,m .
\]
A strict $\vec k$ homomorphism
\[
        f:B \mathbb Z_{\vec k}\longrightarrow  X
\]
is called a $\vec k$-twisted sector if the restriction of $f_1$ on $\langle\lambda_i\rangle, 1\leq i\leq m,$ is injective.  \(f\) determines a point
\[
        x=f_0(\bullet)\in X_0
\]
together with 
$$
(g_1,\ldots, g_m)\in \Gamma_x^m,  
$$
where $g_i=f_1(\lambda_i)$ is of order $k_i$.
Let
\[
        X_{\vec k,0}
        =\{(x,g_1,\ldots, g_m)| \exists f: B\mathbb Z_{\vec k}\to X \mbox{ s.t. } f_0(\bullet)=x, f_1(\lambda_i)=g_i
        \}
\]
denote the space $\vec k$-twisted sectors of $X$.   
This is known to be a  smooth manifold. The projection of $X_{\vec k,0}\to X_0$
\[
(x,g_1,\ldots, g_m)\mapsto x
\]
induces an $X$ action on $X_{\vec k,0}$, set 
\[
X_{\vec k}=X_{\vec,0}\rtimes X.
\]
Define 
\begin{equation}
    X_k=X_{(k)}.
\end{equation}
There are evaluation maps
\begin{equation}
    ev_i: X_{\vec k}\to X_{k_i}\;\;\; \mbox{ such that }\;\;\;
    ev_{i,0}(x,g_1,\ldots,g_m)=(x,g_i).
\end{equation}
The inertia groupoid of $X$ is defined to be 
\[
IX=\bigsqcup_{k\geq 0} X_k
=\bigsqcup_{(g)} X_{(g)},
\]
where $(g)$ is a conjugate class of some group element. Therefore,
\[
        (IX)_0
        =
        \{(x,g)\mid x\in X^0,\ g\in \Gamma_x\},
\]
and
\[
        (IX)_1
        =
        \left\{
        \bigl((x,g),\alpha,(y,\alpha^{-1}g\alpha)\bigr)
        \ \middle|\
        \alpha:x\to y\text{ in }X_1
        \right\}.
\]
There is a natural involution
\[
         J:IX\longrightarrow IX
\]
defined by inversion in the isotropy group:
\[
        J_0(x,g)=(x,g^{-1}),
\]
and
\[
        J_1\bigl((x,g),\alpha,(y,\alpha^{-1}g\alpha)\bigr)
        =
        \bigl((x,g^{-1}),\alpha,(y,\alpha^{-1}g^{-1}\alpha)\bigr).
\]

As a vector space, Chen--Ruan cohomology is the de Rham cohomology of the
inertia orbifold: \def \CR{\mathrm{CR}}
\def \dR{\mathrm{dR}}
\[
        H^\ast_{\CR}(X)
        :=
        H^\ast_{\dR}
        (IX)
        =
        \bigoplus_{k\geq 1}
        H^\ast_{\dR}(X_{(k)}).
\]
We remark that $\Omega^\bullet_\bas(X)=\Omega^\bullet_\dR(X)$.
For the standard graded Chen--Ruan theory, one inserts the age shift: \def \age{\mathrm{age}}
\[
        H^d_{\CR}( X)
        =
        \bigoplus_{(g)}
        H^{d-2\age(g)}_{\dR}(X_{(g)}).
\]

The Chen--Ruan pairing is defined using the involution \(J\).  If
\(\omega_1,\omega_2\) are forms on a sector \(\mathsf X_{(g)}\), then
\[
        \langle \omega_1,\omega_2\rangle
        :=
        \int_{X_{(g)}}^{\orb}
        \omega_1\wedge  J^\ast\omega_2 .
\]
Thus the pairing pairs each sector with its inverse sector.

The Chen--Ruan product is defined using the three-sectors $X_{(\vec g)}$ where
\(
        (\vec g)=(g_1,g_2,g_3).
\)
There are evaluation maps 
\begin{equation}\label{E_ev0}
ev_{(\vec g),0}: (X_{(\vec g)})_0\to X_0 \Longrightarrow ev_{(\vec g)}:
X_{(\vec g)}\to X,
\end{equation}
and
\[
      ev_i:X_{(g_1,g_2,g_3)}
        \longrightarrow
        X_{(g_i)},
        \qquad
        i=1,2,3.
\]
Assume now that \(X\) is equipped with an almost complex structure.
Choose a compatible Hermitian metric.

The map \eqref{E_ev0} is an immersion.  Using the Hermitian metric, we obtain a
splitting
\[
        (ev_{(\vec g),0})^\ast TX_0
        \cong 
        TX_{(\vec g),0}\oplus N_{(\vec g),0}.
\]
Thus, at the orbifold level,
\[
         (ev_{(\vec g)})^\ast TX
        \cong 
        TX_{(\vec g)}\oplus N_{(\vec g)}.
\]
We call $N_{\vec g}$ 
 the normal bundle of the sector $X_{(\vec g)}$. 

 Let \( X_{(g)}\) be a connected component of
\(X_{(k)}\).  The element \(g\) acts on the normal bundle
\(N_{(g)}^0\), and we have an eigenbundle decomposition
\[
        N_{(g)}^0
        =
        \bigoplus_{\theta}
        N_{(g)}^0(\theta),
\]
where \(g\) acts on \(N_{(g)}^0(\theta)\) by multiplication by
\(
        \exp(2\pi i\theta),
        0<\theta<1.
\)

Define the formal fractional normal bundle of 
$X_{(g)}$ to be 
\begin{equation}\label{E_formal_normal}
        N_{(g),\Phi}^0
        :=
        \bigoplus_{\theta}
        \theta\, N_{(g)}^0(\theta) \;\;\;
        \Longrightarrow\;
        N_{(g),\Phi}
        :=\bigoplus_{\theta}
        \theta\,
        N_{(g)}(\theta):=
        \bigoplus_{\theta}
        \theta\,
        \bigl(N_{(g)}^0(\theta)\rtimes \mathsf X\bigr).
\end{equation}

The obstruction bundle on the three-sector is described intrinsically by the
following \(K\)-theoretic formula, due to the work of Chen--Hu,
Jarvis--Kaufmann--Kimura, and Hu--Wang~\cite{ChenHu2006,JarvisKaufmannKimura2007,HuWang2013}:
\[
        ev_1^\ast  N_{(g_1),\Phi}
        +
        ev_2^\ast  N_{(g_2),\Phi}
        +
        ev_3^\ast  N_{(g_3),\Phi}
        -
     N_{(g_1,g_2,g_3)}
\]
is represented by an actual vector bundle over
\( X_{(g_1,g_2,g_3)}\).  We denote this bundle by
\[
\mathcal        O_{(g_1,g_2,g_3)} .
\]
It is the Chen--Ruan obstruction bundle.

Let
\[
        \alpha_i\in H^\ast_{\dR}(\mathsf X_{(g_i)}),
        \qquad
        i=1,2,3.
\]
The Chen--Ruan three-point function is
\[
        \Psi(\alpha_1,\alpha_2,\alpha_3)
        =
        \int_{\mathsf X_{(g_1,g_2,g_3)}}^{\orb}
        ev_1^\ast\alpha_1
        \wedge
        ev_2^\ast\alpha_2
        \wedge
        ev_3^\ast\alpha_3
        \wedge
        e\bigl(\mathcal O_{(g_1,g_2,g_3)}\bigr).
\]
Here
\(
        e(\mathcal O_{(g_1,g_2,g_3)})
\)
is the Euler class of the obstruction bundle.

The Chen--Ruan product is characterized by the identity
\[
        \langle
        \alpha_1\star_{\CR}\alpha_2,\alpha_3
        \rangle
        =
        \Psi(\alpha_1,\alpha_2,\alpha_3).
\]
Therefore the ring structure on
\(
        H^\ast_{\CR}(\mathsf X)
\)
is determined by the sector pairing and the three-point functions.

\subsection{The induced weak \(G\)-action on the inertia groupoid}

Let
\[
        p:Y\longrightarrow BG
\]
be a Kan fibration of Lie groupoids with fiber \(X\) being a proper \'{e}tale Lie groupoid, so that
\(X\) presents an orbifold. 

Recall that the inertia groupoid \(IX\) has object space
\[
        (IX)_0
        =
        \{(x,g)\mid x\in X_0,\ g\in \Gamma_x\}.
\]
The peojection map 
\[
ev: (IX)_0\to X_0=Y_0
\]
not only admits an $X$ action but also an $Y$ action. Set 
\[
Y_I=(IX)_0\rtimes Y=((Y_I)_1\rrto (Y_I)_0).
\]
We explain the $Y$ action on $(IX)_0$. Let $\gamma:x\to y$ be an arrow in $Y$, it maps $(x,g)$ 
to $(y, h)$ where $h=\gamma^{-1}g\gamma)$. We should verify that $h\in \Gamma_y$. In fact
\[
p_1(h)=p_1(\gamma^{-1})p_1(g)p_1(\gamma)=e.
\]
Therefore, $h\in X_1$ and it is an arrow from $y$ to $y$, so $h\in \Gamma_y$. 
Next, suppose 
\[
x_1\xrightarrow{\gamma_1}x_2\xrightarrow{\gamma_2}x_3
\]
are arrows in $Y$ and $\gamma_1\gamma_2:x_1\to x_3 $. Then 
\begin{equation*}
(x_1,g_1)\xrightarrow{\gamma_1} (x_2, \gamma_1^{-1}g\gamma_1)\xrightarrow{\gamma_2} (x_2, \gamma_2^{-1}\gamma_1^{-1}g\gamma_1\gamma_2)
\end{equation*}
agrees with 
\[
(x_1,g_1)\xrightarrow{\gamma_1\gamma_2} (x_3, (\gamma_1\gamma_2)^{-1}g(\gamma_1\gamma_2)).
\]
This verifies the $Y$ action on $(IX)_0$. Hence, 
\[
(Y_I)_0=(IX)_0
\]
and
\[
(Y_I)_1=(IX)_0\times_{ev,Y_0,s_Y} Y_1. 
\]

There is a natural homomorphism
\[
        p_I:Y_I\longrightarrow BG
\]
defined on objects by the unique map to the object of \(BG\), and on arrows by
\[
        (p_I)_1(\gamma,(x,g)):=p_1(\gamma).
\]

\begin{prop}
The homomorphism
\[
        p_I:Y_I\longrightarrow BG
\]
is a Kan fibration.  Its fiber is the inertia groupoid \(IX\).
\end{prop}
The proof is straightforward, we omit the details. Here 
We give a formal explanation. 
Since 
\[
Y_I=(IX)_0\rtimes Y
=(IX)_0\rtimes (X\rtimes BG)
=((IX)_0\rtimes X)\rtimes BG=IX\rtimes BG,
 \]
we have $p_I: Y_I\to BG$ whose fiber is $IX$.

Since $IX$ is  an orbifold with the $G$ action, we have 
\begin{equation}
    H^d_{G,CR}(X)=
     \bigoplus_{(g)}
        H^{d-2\age(g)}_{G,\dR}(X_{(g)}).
\end{equation}

\subsection{Equivariant Chen--Ruan cohomology}\label{Subsec_equiv_3pt}

Let \(\mathsf X\) be an almost complex orbifold equipped with a weak
\(G\)-action given by the Kan fibration \eqref{E_KF}.
We may modify the almost complex structure such that it is $G$-invariant. Here we use the fact that $G$ is compact.

A clean equivariant three-point function is the ordinary Chen--Ruan three-point function with every ingredient replaced by its G-equivariant version: equivariant sector cohomology, equivariant evaluation pullbacks, equivariant obstruction Euler class, and equivariant orbifold integration.

For equivariant classes
\[
        \alpha_i\in
        H^\bullet_{G,\dR}\bigl(\mathsf X_{(g_i)}\bigr),
        \qquad i=1,2,3,
\]
define the sectorwise equivariant Chen--Ruan three-point function by
\[
        \Psi_{G,(g_1,g_2,g_3)}
        (\alpha_1,\alpha_2,\alpha_3)
        :=
        \int_{ X_{(g_1,g_2,g_3)}}^{\orb,G}
        ev_{1}^*\alpha_1
        \wedge
        ev_2^*\alpha_2
        \wedge
        ev_3^*\alpha_3
        \wedge
        e_G\bigl(\mathcal O_{(g_1,g_2,g_3)}\bigr).
\]
The value lies in
\[
        H^\bullet_G(\mathrm{pt})
        =
        S(\mathfrak g^*)^G .
\]
This naturally defines an equivariant three-point function 
\begin{equation}
    \Psi_G: \bigl( H^\bullet_{G,CR}(IX)\bigr)^{\otimes 3}\to S(\mathfrak g^\ast)^G.
\end{equation}

The equivariant Chen--Ruan pairing is
\begin{equation}\label{E_PP}
        \langle \alpha,\beta\rangle_G
        :=
        \int_{{IX}}^{\orb,G}
        \alpha\wedge  J_G^*\beta ,
\end{equation}
where
\[
        J_G:
        {IX}\rtimes \mathsf B G
        \longrightarrow
        {IX}\rtimes \mathsf B G
\]
is the \(G\)-equivariant extension of the involution
$J: IX\to IX$.

The equivariant Chen--Ruan product
\(
        \star_{\CR,G}
\)
is characterized by
\[
        \left\langle
        \alpha_1\star_{\CR,G}\alpha_2,
        \alpha_3
        \right\rangle_G
        =
        \Psi_G(\alpha_1,\alpha_2,\alpha_3)
\]
if the pairing \eqref{E_PP} is non-degenerated.

We remark that the equivariant Poincare duality may not hold for arbitrary manifold/orbifold, while the equivariant three-point function is well defined.

\section*{Statements and Declarations}

\subsection*{Funding}
This work was supported by the National Key R\&D Program of China (No.~2020YFA0714000).

\subsection*{Competing interests}
The authors have no competing interests to declare.

\subsection*{Data availability}
Data sharing is not applicable to this article as no datasets were generated or analyzed during the current study.
\makeatletter\immediate\write\@auxout{\string\citation{bst-control}}\makeatother
\bibliography{references}

@settings{bst-control,
  options = {sort},
  key = {bst-control}
}

@article{Behrend2005,
  author  = {Behrend, Kai},
  title   = {On the de {R}ham cohomology of differential and algebraic stacks},
  journal = {Advances in Mathematics},
  volume  = {198},
  number  = {2},
  pages   = {583--622},
  year    = {2005},
  doi     = {10.1016/j.aim.2005.05.025}
}

@article{BehrendXu2011,
  author  = {Behrend, Kai and Xu, Ping},
  title   = {Differentiable stacks and gerbes},
  journal = {Journal of Symplectic Geometry},
  volume  = {9},
  number  = {3},
  pages   = {285--341},
  year    = {2011},
  doi     = {10.4310/JSG.2011.v9.n3.a2}
}

@article{Blohmann2008,
  author  = {Blohmann, Christian},
  title   = {Stacky {L}ie groups},
  journal = {International Mathematics Research Notices},
  year    = {2008},
  pages   = {rnn082},
  doi     = {10.1093/imrn/rnn082}
}

@article{BottShulmanStasheff1976,
  author  = {Bott, Raoul and Shulman, Harold and Stasheff, James D.},
  title   = {On the de {R}ham theory of certain classifying spaces},
  journal = {Advances in Mathematics},
  volume  = {20},
  number  = {1},
  pages   = {43--56},
  year    = {1976},
  doi     = {10.1016/0001-8708(76)90169-9}
}

@article{BursztynNosedaZhu2020,
  author  = {Bursztyn, Henrique and Noseda, Francesco and Zhu, Chenchang},
  title   = {Principal actions of stacky {L}ie groupoids},
  journal = {International Mathematics Research Notices},
  year    = {2020},
  number  = {16},
  pages   = {5055--5125},
  doi     = {10.1093/imrn/rny142}
}

@article{ChenHu2006,
  author  = {Chen, Bohui and Hu, Shengda},
  title   = {A de {R}ham model for {C}hen--{R}uan cohomology ring of abelian orbifolds},
  journal = {Mathematische Annalen},
  volume  = {336},
  pages   = {51--71},
  year    = {2006},
  doi     = {10.1007/s00208-006-0774-3}
}

@article{ChenRuan2004,
  author  = {Chen, Weimin and Ruan, Yongbin},
  title   = {A new cohomology theory of orbifold},
  journal = {Communications in Mathematical Physics},
  volume  = {248},
  pages   = {1--31},
  year    = {2004},
  doi     = {10.1007/s00220-004-1089-4}
}

@article{delHoyo2013,
  author  = {{del Hoyo}, Matias L.},
  title   = {Lie groupoids and their orbispaces},
  journal = {Portugaliae Mathematica},
  volume  = {70},
  number  = {2},
  pages   = {161--209},
  year    = {2013},
  doi     = {10.4171/PM/1930}
}

@phdthesis{DuLi2014,
  author = {Li, Du},
  title  = {Higher Groupoid Actions, Bibundles, and Differentiation},
  school = {Georg-August-Universit{\"a}t G{\"o}ttingen},
  year   = {2014},
  url    = {https://arxiv.org/abs/1512.04209}
}

@article{EdidinJarvisKimura2010,
  author  = {Edidin, Dan and Jarvis, Tyler J. and Kimura, Takashi},
  title   = {Logarithmic trace and orbifold products},
  journal = {Duke Mathematical Journal},
  volume  = {153},
  number  = {3},
  pages   = {427--473},
  year    = {2010},
  doi     = {10.1215/00127094-2010-028}
}

@article{FantechiGottsche2003,
  author  = {Fantechi, Barbara and G{\"o}ttsche, Lothar},
  title   = {Orbifold cohomology for global quotients},
  journal = {Duke Mathematical Journal},
  volume  = {117},
  number  = {2},
  pages   = {197--227},
  year    = {2003},
  doi     = {10.1215/S0012-7094-03-11721-4}
}

@misc{GinotNoohi2012,
  author = {Ginot, Gr{\'e}gory and Noohi, Behrang},
  title  = {Group actions on stacks and applications to equivariant string topology for stacks},
  year   = {2012},
  note   = {arXiv:1206.5603},
  url    = {https://arxiv.org/abs/1206.5603}
}

@article{GoldinHolmKnutson2007,
  author  = {Goldin, Rebecca and Holm, Tara S. and Knutson, Allen},
  title   = {Orbifold cohomology of torus quotients},
  journal = {Duke Mathematical Journal},
  volume  = {139},
  number  = {1},
  pages   = {89--139},
  year    = {2007},
  doi     = {10.1215/S0012-7094-07-13912-7}
}

@book{GuilleminSternberg1999,
  author    = {Guillemin, Victor and Sternberg, Shlomo and Br{\"u}ning, Jochen},
  title     = {Supersymmetry and Equivariant de {R}ham Theory},
  series    = {Mathematics Past and Present},
  publisher = {Springer},
  address   = {Berlin},
  year      = {1999},
  doi       = {10.1007/978-3-662-03992-2}
}

@article{HilsumSkandalis1987,
  author  = {Hilsum, Michel and Skandalis, Georges},
  title   = {Morphismes {$K$}-orient{\'e}s d'espaces de feuilles et fonctorialit{\'e} en th{\'e}orie de {K}asparov},
  journal = {Annales Scientifiques de l'{\'E}cole Normale Sup{\'e}rieure},
  volume  = {20},
  number  = {3},
  pages   = {325--390},
  year    = {1987},
  doi     = {10.24033/asens.1537}
}

@article{HolmMatsumura2012,
  author  = {Holm, Tara S. and Matsumura, T.},
  title   = {Equivariant cohomology for {H}amiltonian torus actions on symplectic orbifolds},
  journal = {Transformation Groups},
  volume  = {17},
  number  = {3},
  pages   = {717--746},
  year    = {2012},
  doi     = {10.1007/s00031-012-9192-7}
}

@article{HuWang2013,
  author  = {Hu, Jianxun and Wang, Bai-Ling},
  title   = {Delocalized {C}hern character for stringy orbifold {$K$}-theory},
  journal = {Transactions of the American Mathematical Society},
  volume  = {365},
  number  = {12},
  pages   = {6309--6341},
  year    = {2013},
  doi     = {10.1090/S0002-9947-2013-05834-5}
}

@article{JarvisKaufmannKimura2007,
  author  = {Jarvis, Tyler J. and Kaufmann, Ralph and Kimura, Takashi},
  title   = {Stringy {$K$}-theory and the {C}hern character},
  journal = {Inventiones Mathematicae},
  volume  = {168},
  number  = {1},
  pages   = {23--81},
  year    = {2007},
  doi     = {10.1007/s00222-006-0026-x}
}

@article{Lerman2010,
  author  = {Lerman, Eugene},
  title   = {Orbifolds as stacks?},
  journal = {L'Enseignement Math{\'e}matique},
  volume  = {56},
  number  = {3--4},
  pages   = {315--363},
  year    = {2010},
  doi     = {10.4171/LEM/56-3-4}
}

@book{Mackenzie2005,
  author    = {Mackenzie, Kirill C. H.},
  title     = {General Theory of {L}ie Groupoids and {L}ie Algebroids},
  publisher = {Cambridge University Press},
  address   = {Cambridge},
  year      = {2005},
  doi       = {10.1017/CBO9781107325883}
}

@article{MathaiQuillen1986,
  author  = {Mathai, Varghese and Quillen, Daniel},
  title   = {Superconnections, {T}hom classes, and equivariant differential forms},
  journal = {Topology},
  volume  = {25},
  number  = {1},
  pages   = {85--110},
  year    = {1986},
  doi     = {10.1016/0040-9383(86)90007-8}
}

@book{MoerdijkMrcun2003,
  author    = {Moerdijk, Ieke and Mr{\v{c}}un, Janez},
  title     = {Introduction to Foliations and {L}ie Groupoids},
  series    = {Cambridge Studies in Advanced Mathematics},
  volume    = {91},
  publisher = {Cambridge University Press},
  address   = {Cambridge},
  year      = {2003},
  doi       = {10.1017/CBO9780511615450}
}

@article{PflaumPosthumaTang2014,
  author  = {Pflaum, Markus J. and Posthuma, Hessel and Tang, Xiang},
  title   = {Geometry of orbit spaces of proper {L}ie groupoids},
  journal = {Journal f{\"u}r die reine und angewandte Mathematik},
  volume  = {694},
  pages   = {49--84},
  year    = {2014},
  doi     = {10.1515/crelle-2012-0092}
}

@incollection{Watts2022,
  author    = {Watts, Jordan},
  title     = {The orbit space and basic forms of a proper {L}ie groupoid},
  booktitle = {Current Trends in Analysis, its Applications and Computation},
  series    = {Trends in Mathematics: Research Perspectives},
  publisher = {Birkh{\"a}user},
  address   = {Cham},
  pages     = {513--523},
  year      = {2022},
  doi       = {10.1007/978-3-030-87502-2_52}
}

@unpublished{ChenDuJiangSubmitted,
  author = {Chen, Bohui and Du, Cheng-Yong and Jiang, Fengyu},
  title  = {Automorphisms of Lie Groupoids and Symplectic Reduction on Orbifolds},
  note   = {Submitted manuscript},
  year   = {2025}
}

\end{document}